\documentclass[titlepage, 12pt]{article}
\usepackage[margin=1in]{geometry} 
\usepackage{amssymb} 
\usepackage{amsmath} 
\usepackage{amsthm} 
\usepackage{graphicx} 
\usepackage{booktabs} 
\usepackage{natbib} 
\usepackage{pgfplots} 
\usepgfplotslibrary{fillbetween} 
\usepackage{bm} 
\usepackage{xcolor} 
\definecolor{royalblue}{HTML}{0000CD} 
\usepackage[colorlinks=true]{hyperref} 
\hypersetup{
    colorlinks=true,
    allcolors=royalblue,
} 
\usepackage{threeparttable} 
\usepackage{thmtools} 
\declaretheorem[numberwithin=section]{theorem} 
\newtheorem{corollary}{Corollary}[theorem] 
\allowdisplaybreaks 
\usepackage{mathtools} 

\newtheorem{Lemma}{{Lemma}}

\theoremstyle{definition}
\newtheorem*{remark}{Remark}

\providecommand{\keywords}[1]
{
  \noindent\small	
  \textbf{\textit{Keywords:}} #1
}

\usepackage{caption} 
\usepackage{subcaption} 

\usepackage{setspace} 

\begin{document}

\title{Estimating a discrete distribution subject to random left-truncation
with an application to structured finance\thanks{This material is
based upon work supported by the National Science Foundation
Graduate Research Fellowship under Grant No. DHE 1747453.}}

\author{
  Jackson P. Lautier\footnote{Department of Statistics, University of
  Connecticut}%
  \thanks{Corresponding to jackson.lautier@uconn.edu.}
  \and
  Vladimir Pozdnyakov\footnotemark[2]
  \and
  Jun Yan\footnotemark[2]
}

\date{\today}

\maketitle

\begin{abstract}
Proper econometric analysis should be informed by data structure.  Many forms
of financial data are recorded in discrete-time and relate to products of a 
finite term.  If the data comes from a financial trust, it will often be 
further subject to random left-truncation. While the literature for estimating
a distribution function from left-truncated data is extensive, a thorough 
literature search reveals that the case of discrete data over a finite number of 
possible values has received little attention.  A precise discrete framework
and suitable sampling procedure for the Woodroofe-type estimator for discrete
data over a finite number of possible values is therefore established. 
Subsequently, the resulting vector of hazard rate estimators is proved to be
asymptotically normal with independent components. Asymptotic normality of
the survival function estimator is then established. Sister results for the 
left-truncating random variable are also proved. Taken together, the resulting 
joint vector of hazard rate estimates for the lifetime and left-truncation
random variables is proved to be the maximum likelihood estimate
of the parameters of the conditional joint lifetime and left-truncation
distribution given the lifetime has not been left-truncated.
A hypothesis test for the shape of 
the distribution function based on our asymptotic results is derived. Such a 
test is useful to formally assess the plausibility of the stationarity 
assumption in length-biased sampling.  The finite sample performance of the 
estimators is investigated in a simulation study.  Applicability of the
theoretical results in an econometric setting is demonstrated with a subset of
data from the Mercedes-Benz 2017-A securitized bond.

\bigskip

\keywords{asset-backed security, asset-level disclosures, consumer lease
securitization, product-limit estimator, reverse hazard rate, Reg AB II}

\end{abstract}

\doublespacing

\section{Introduction}
\label{sec:intro}

The current outstanding issuance of consumer auto lease asset-backed securities
(ABS) in the United States is nearly \$35 billion \citep{sifma_2022}, and the
recent implementation of Reg AB II \citep{cfr_229} has made a glut of public
asset level ABS data available to investors for the first time.  While more
transparency into the underlying assets is generally a benefit for investors, 
the data may be difficult to analyze. This is because the legal structure
of an ABS trust, the terms of a standard consumer automobile lease contract, 
and the nature of a monthly due date creates a need to consider 
left-truncation, a finite time horizon, and discrete-time, respectively,
in estimating the
distribution of consumer lease lifetimes. We elaborate with a specific example
from structured finance. Consider an automotive lease securitization, such as
\citet{mercedes_2017}, in which consumer automotive lease contracts are pooled
together into a trust. Standard automotive lease contracts have a fixed and 
known duration, such as 36 months, with required monthly payments. Further, the
payment performance of the lessee will be reported monthly, so the observed 
survival times of the lease contracts will be discrete within the nonnegative integers, $\mathbb{N}$.
Left-truncation occurs because only those leases that remain active long enough
to be collected into the trust will be observable by the investor. In the 
literature of survival analysis, this is a form of bias under the general 
umbrella of delayed entry or length-biased sampling \citep[e.g.,][]{
asgharian_2002, una_alvarez_2004, asgharian_2005, huang_2011}.

To formalize, let $X$ denote the random time of a lease contract termination
(i.e., the lifetime or time-to-event random variable) and let $T$ denote the 
random time of a lease contract origination.  The context of our
application naturally restricts $X$ and $T$ to a finite subset of consecutive
integers. If $\omega \in \mathbb{N}$ represents the age
of the last lease termination in a sample, then $X \leq \omega$.  Since issuers
of structured debt typically have a legal obligation to the trust to select
lease contracts with a minimum history of on-time payments, the youngest
lease in the trust will have a minimum age of $\Delta$ as of the onset of
the trust, where $\Delta \in \mathbb{N}$.  Hence, each lease will
have a minimum survival time of $\Delta + 1$, and so $\Delta + 1 \leq X
\leq \omega$.  If $m \in \mathbb{N}$ is the origination time of the youngest
lease in the trust, then $1 \leq T \leq m$ and the trust starting time is
$m + \Delta$.  For all practical purposes, $m + \Delta \leq \omega$.  
The integers $\Delta, m$, and $\omega$ are
non-random and known as of the onset of the problem. Notably,
if we define $Y = m + \Delta + 1 - T$, then $Y$ denotes a left-truncation
random variable representing the minimum amount of time a lease must remain
active to be observed in the trust.  In other words, an investor will only
observe those leases such that $X \geq Y$.  For completeness,
$\Delta + 1 \leq Y \leq \Delta + m$.  We present a visualization of the
connected random variables and timelines in Figure~\ref{fig:number_line}.
Throughout, we assume $X$ and $T$ are independent (and
therefore $X$ and $Y$ are independent).

\begin{figure}[tbh!]
    \begin{subfigure}[tbp]{1\textwidth}
         \centering
		\begin{tikzpicture}
		\begin{axis}[
  		height=2cm,
  		width=10cm,
  		axis y line=none,
  		axis lines=left,
  		axis line style={-},
  		xmin=1,
  		xmax=10,
  		ymin=0,
  		ymax=1,
  		xlabel={\text{}},
  		scatter/classes={o={mark=*}},
  		restrict y to domain=0:1,
  		xtick={1,2, 5, 9,10},
  		xticklabels={1, ,$T$, ,$m$},
		]
		\end{axis}
		\end{tikzpicture}
		\caption{Symbolically, $T$ represents a random lease start time.
		Nonrandom time $m$ is the origination time of the youngest lease
		in the trust as of the beginning of the trust observation window.
		Thus, $1 \leq T \leq m$, where $T,m \in \mathbb{N}$.}
    \end{subfigure}
    \par\bigskip
    \begin{subfigure}[tbp]{1\textwidth}
    		\centering
		\begin{tikzpicture}
		\begin{axis}[
  		height=2cm,
  		width=10cm,
  		axis y line=none,
  		axis lines=left,
  		axis line style={-},
  		xmin=1,
  		xmax=10,
  		ymin=0,
  		ymax=1,
  		xlabel={\text{}},
  		scatter/classes={o={mark=*}},
  		restrict y to domain=0:1,
  		xtick={1,2, 5, 9,10},
  		xticklabels={$\Delta+1$, ,$m+\Delta+1-T$, ,$\Delta+m$},
  		clip=false
		]
		\node[coordinate,label=above:{$Y$}] at (axis cs:5,0.15) {};
		\end{axis}
		\end{tikzpicture}
		\caption{We call the time that the trust observation window begins
		$\Delta + m$, and so nonrandom $\Delta$ denotes the minimum age of a
		lease in the trust as of time $\Delta + m$.  Defining $Y = m + \Delta
		+ 1 - T$ with $1 \leq T \leq m$
		implies $\Delta + 1 \leq Y \leq \Delta + m$, where 
		$Y, m, \Delta, T \in \mathbb{N}$.}
	\end{subfigure}
	\par\bigskip
	\begin{subfigure}[tbp]{1\textwidth}
		\centering
		\begin{tikzpicture}
		\begin{axis}[
  		height=2cm,
  		width=10cm,
  		axis y line=none,
  		axis lines=left,
  		axis line style={-},
  		xmin=1,
  		xmax=10,
  		ymin=0,
  		ymax=1,
  		xlabel={\text{}},
  		scatter/classes={o={mark=*}},
  		restrict y to domain=0:1,
  		xtick={1,2, 5, 9,10},
  		xticklabels={$Y$, ,$X$, ,$\omega$},
		]
		\end{axis}
		\end{tikzpicture}
		\caption{We only observe the random lease termination time, $X$, if
		$X \geq Y$.  Nonrandom $\omega$ represents the termination time of the
		lease with the longest active ongoing payments, and it coincides
		with the close of the trust observation window.  Thus, $\Delta + 1
		\leq X \leq \omega$, where $X, \Delta, \omega, Y \in \mathbb{N}$.}
	\end{subfigure}
	\caption{The connected discrete random variables $T$, $Y$, $X$ and the
	associated finite timelines for left-truncated data from an auto lease
	securitization.}
    \label{fig:number_line}
\end{figure}

The classical problem of estimating a distribution function in the presence of
random left-truncation has a sizable history in the statistical literature.  
Specifically, if we consider two independent positive random variables $X$ and 
$Y$ with distribution functions $F$ and $G$ such that we only observe the pairs
$(X, Y)$ for which $Y \leq X$ and the pairs $(X,Y)$ are assumed to be 
independent and identically distributed (i.i.d.), it is not difficult to find many
thorough studies \citep[e.g.,][]{lynden_1971, woodroofe_1985, wang_1986, 
keiding_1990, stute_1993, he_1998b}.  However, the nature of securitization
data requires us to further assume that $X$ and $Y$ are nonnegative 
integer-valued random variables with a finite number of possible values (though 
the remaining assumptions of the classical problem remain valid).

To our surprise, a thorough literature review revealed that the case of discrete
$F$ and $G$ have received little attention.  
Two seminal works in this field are \citet{woodroofe_1985}
and \citet{wang_1986}. \citet{woodroofe_1985} proves consistency results for 
the \citet{lynden_1971} estimator and shows its weak convergence to a Gaussian 
process but left the exact form of the covariance structure of the limiting 
process undefined. In deriving the asymptotic results, 
\citet{woodroofe_1985} assumes 
continuous distribution functions $F$ and $G$.  \citet{wang_1986} extends the 
results of \citet{woodroofe_1985} with a precise description of the asymptotic
covariance structure. It is noteworthy that this structure is the analogue of 
the covariance structure of the Kaplan--Meier estimator.  \citet{wang_1986} 
alludes to the idea that $F$ and $G$ need not be continuous in establishing 
strong consistency for the product limit estimator of $F$, but they assume 
continuity of $F$ and $G$ in working to define the covariance structure.
Since \citet{woodroofe_1985} and \citet{wang_1986}, there has been many
notable and significant contributions; interested readers may find
a thorough literature review in Appendix~\ref{sec:lit_review}.

Typical approaches for avoiding an assumption of discrete-time may be
problematic or inappropriate for consumer lease ABS data.  First, one
may be tempted to force an assumption of continuous $F$ and $G$, but this
implies that ties are events with zero probability. Since there are likely many
lease contracts with the same termination time, this creates an immediate 
complication.  Second, one may treat the lease performance data as 
interval-censored, where the event is assumed to occur within an interval of
time (i.e., a month) but the exact time within the interval is unknown.  With
lease contracts and loan contracts more generally, however, payments made prior
to a due date are treated the same as payments made on the due date 
(prepayments aside).  In other words, a monthly payment was either received 
on-time or is delinquent.  Thus, the ABS data is in actuality discrete with 
jumps at integer intervals; it is not a product of imprecise measurements.  For
similar reasons, even standard grouped survival data approaches 
\citep[e.g.,][]{prentice_1978} are not true representations of the failure time
random variable for consumer lease data. One technical remark is that
if the distribution function is known to contain jumps, but 
the location of such jumps is not known prior to performing the estimation, the
analysis is subject to additional complications, as in \citet{rabhi_2017}.  We 
are working over $\mathbb{N}$ and thus may avoid this potentially 
cumbersome framework.

Overall, our contributions are thus.  We fill the unexpected gap in the literature for the discrete case of $X$ and $Y$
for the Woodroofe-type estimator.  The first main result is
that the vector of these estimators in discrete-time over a
finite number of possible values is asymptotically normal
with a fully-specified diagonal covariance matrix.  The
second main result is that the vector of Woodroofe-type
estimates is the MLE for the parameters of the
discrete bivariate distribution of $(X,Y)$ given
$Y \leq X$.  We also find
these results have a significant application potential
for the large fixed-income asset class of consumer lease
ABS data.  A detailed outline is as follows.
In Section~\ref{sec:est}, we precisely define
the joint conditional discrete sample space for $X$ and $Y$, the related
discrete conditional bivariate probability mass function and its connection to 
$F$ and $G$ through the hazard and \textit{reverse hazard} rates, respectively, 
and a suitable sampling procedure to mimic the realities of securitized trust
data (this sampling process differs from \citet{woodroofe_1985}).  These
preliminaries are necessary because the discrete case has not before received a
rigorous treatment in the literature. Section~\ref{sec:est} closes by 
presenting the estimators and the first major result: all together,
these estimates
are the MLE of the discrete conditional bivariate distribution.  
Section~\ref{sec:asym_res} provides the next set of major results in that
we state the asymptotic normality and independence of the estimation vector of
the hazard rates for $F$ and its analog for $G$; and, asymptotic normality of
the estimator for the survival function of $X$ and the estimator of the 
distribution function of $Y$.  In all cases, the diagonal covariance matrix is
completely specified.  We also derive a hypothesis test for the shape
of the distribution function with applications to length-biased sampling. 
In Section~\ref{sec:sim}, we experimentally validate the results in 
Section~\ref{sec:asym_res} with a simulation study. In Section~\ref{sec:app},
we apply our results to a sample of data from the Mercedes Benz 2017-A 
securitized bond \citep{mercedes_2017}.  The paper closes with a brief 
discussion.  Appendix~\ref{sec:lit_review} presents a thorough literature 
review, and Appendix~\ref{sec:proofs} provides complete proofs of all major
results.

\section{Estimation}
\label{sec:est}

We begin by briefly reviewing notation and the identifiability results
from \citet{woodroofe_1985}.  Establishing the discrete sample space begins
in the next section on recovery. This requires defining
a discrete conditional bivariate probability mass function and related 
conditional distributions for discrete $X$ and $Y$.  Because of the discrete
nature of $X$ and $Y$, it is preferable to work in terms of the hazard rate of
$X$ and the reverse hazard rate of $Y$ (for continuous $F$ and $G$, the
cumulative hazard function works well).  We then connect the hazard and 
reverse hazard rates to $F$ and $G$, respectively, in the discrete case.
Estimators for both the hazard and reverse hazard rates are then formally 
defined in the context of sampling from a left-truncated population rather than 
left-truncating a joint random sample (a further distinction from 
\citet{woodroofe_1985}).  Finally, we formally state the result that 
the joint vector of
estimates for the hazard and reverse hazard rates is a MLE for the discrete
conditional bivariate probability distribution for $(X,Y)$ given $X \geq Y$. 

\subsection{Preliminaries}

Working from the notation of \citet{woodroofe_1985}, let $F$ and $G$
be the distribution functions of non-negative independent random
variables $X$ and $Y$, respectively. Let $H_*$ denote the joint
distribution function of $X$ and $Y$ given $Y \leq X$, and let $F_*$
and $G_*$ denote the marginal distributions functions given $Y \leq X$
of $X$ and $Y$, respectively.  That is,
\begin{equation*}
    H_*(F,G,x,y) = \Pr(X \leq x, Y \leq y \mid X \geq Y),
\end{equation*}
is the joint conditional distribution function with conditional
marginal distributions $F_*$ and $G_*$.  We include $F$ and $G$ within
the definition of $H_*$ to stress which $F$ and $G$ are employed to
construct $H_*$. For convenience, we may drop $x$ and
$y$ from the notation for $H_*$ when the meaning is clear or if the
clarification is nonessential; i.e., $H_*(F,G)$.

We now review key observations made by \citet{woodroofe_1985}.
Define
\begin{equation*}
    a_F = \inf \{ z > 0 : F(z) > 0 \} \geq 0,
\end{equation*}
and
\begin{equation*}
    b_F = \sup \{z > 0 : F(z) < 1 \} \leq \infty.
\end{equation*}
That is, $(a_F, b_F)$ is the interior of the convex support of $F$ and
similarly $(a_G, b_G)$ for $G$.  To avoid complete left-truncation
and full data loss, we must have $a_G < b_F$.

Next, we need to introduce two classes of distribution pairs $(F,G)$.
The first class includes all pairs of $F$ and $G$ that allow the construction
of the two-dimensional distribution $H_*$,
\begin{equation*}
    \mathcal{K} = \{ (F, G) : F(0) = 0 = G(0), \quad \Pr(Y \leq X) > 0 \}.
\end{equation*}
The second class includes those pairs $(F,G)$ that can be recovered from $H$,
\begin{equation*}
  \mathcal{K}_0 = \{ (F,G) \in \mathcal{K} : a_G \leq a_F, \quad b_G \leq b_F \}.
\end{equation*}
\citet{woodroofe_1985} demonstrated in his Lemma~1
that if we take any $(F,G) \in \mathcal{K}$ and let $F_0 = \Pr(X \leq x
\mid X \geq a_G)$ and $G_0 = \Pr(Y \leq y \mid Y \leq b_F)$,  then
$(F_0, G_0) \in \mathcal{K}_0$ and $H_*(F_0, G_0) = H_*(F,G)$.
This subtle but important result implies that, if given $H_*$,
we may not be able to recover the pair $(F,G)$.  This is because there is
another pair, $(F_0,G_0)$, that gives us exactly the same $H_*$. It is not
surprising. For example, in the context of our motivating problem, we only
observe $X$ when it is equal or greater than  $\Delta+1$. Hence, it is
impossible to get any information on the distribution of $X$ over values less
than $\Delta+1$.

\subsection{Recovery}

\citet{woodroofe_1985} shows in his Theorem 1 that if we restrict our
construction of $H_*$ to class $\mathcal{K}_0$, then this operation is
``invertible''. More specifically, for every $H$ based on some
$(F,G)\in \mathcal{K}$ there is only one pair $(F_0,G_0)\in  \mathcal{K}_0$
such that $H_*(F_0, G_0) = H_*(F,G)$  and this pair is given by
$F_0$ and $G_0$. Moreover, this theorem gives specific instructions
on how to recover the cumulative hazard functions  of $F_0$ and $G_0$
(and, therefore, $F_0$ and $G_0$ as well).

Once again, in the context of our example, we have
$F_0(x) = \Pr(\Delta + 1 \leq X \leq x) / \Pr(X \geq \Delta + 1)$
(that is, the range of $F_0$ is $\{\Delta + 1, \ldots, \omega\}$) and
$G_0(y) = \Pr(Y \leq y)=G(y)$ because $\Delta + m \leq \omega$ by assumption.
The range of $G_0$ is $\{\Delta + 1, \ldots, \Delta + m\}$. Thus, from
$H_*$ based on the original $F$ and $G$, it is possible to recover $G$ but
only the $F_0$ portion of $F$.

We have discussed $X$ and $Y$ at length thus far, but we now do so with
some additional precision.  Specifically, let $X \in \mathbb{N}$ and
$Y \in \mathbb{N}$ be independent random variables with ranges
$\{\Delta+1, \ldots, \omega\}$ and $\{\Delta+1, \ldots, \Delta+m\}$,
respectively. We will assume 
that $\Pr(X=\Delta+1)$, $\Pr(Y=\Delta+1)$, $\Pr(X=\omega)$, and 
$\Pr(Y=\Delta+m)$ are strictly positive, and $\Delta+m\leq \omega$. Let $A$ 
be a set of points on the plane $\mathbb{N} \times \mathbb{N}$ with
integer-valued coordinates $(u,v)$ such that
$u\in \{\Delta+1, \ldots, \omega\}$,
$v \in \{\Delta+1, \ldots, \Delta+m\}$, and $v\leq u$.  A visualization of
$A$ may be found in Figure~\ref{fig:trapezoid}.

\begin{figure}[tbh!]
\centering
\begin{tikzpicture}
  \begin{axis}[
  	height=8cm,
  	width=12cm,
    axis lines=middle,
    xlabel=$x$,ylabel=$y$,
    xmin=0,xmax=6.5,ymin=0,ymax=6,
    xtick={1,2,3,4,5,6},
    xticklabels={$\Delta + 1$, $\Delta + 2$, ,$\Delta + m$ , ,$\omega$},
    ytick={1,2,3,4},
    yticklabels={$\Delta + 1$, $\Delta + 2$, , $\Delta + m$}
    ]
    \path[name path=axis] (axis cs:1,1) -- (axis cs:1,6);
    \addplot+[only marks, mark options={fill=black,draw=black}] coordinates {
    (1,1) (2,2) (3,3) (4,4) (2,1) (3,1) (4,1) (5,1) (6,1)
    (3,2) (4,2) (5,2) (6,2)
    (4,3) (5,3) (6,3)
    (5,4) (6,4)};
    \addplot[name path=f] {x}
    node [pos=1, below right] {$y=x$};
    \addplot [mark=none,draw=none,name path=bottom] plot coordinates {
    (1,1)
    (6,1)
    };
    \addplot [mark=none,draw=none,name path=right] plot coordinates {
    (6,1)
    (6,4)
    };
    \addplot [mark=none,draw=none,name path=top] plot coordinates {
    (4,4)
    (6,4)
    };
    \addplot[gray!30] fill between[of=bottom and f,soft clip={domain=1:4.01}];
    \addplot[gray!30] fill between[of=bottom and top,soft clip={domain=4:6}];
    \end{axis}
\end{tikzpicture}
\caption{The set of points on the plane with $(u,v) \in \mathbb{N}$ such that 
$u \in \{\Delta+1, \ldots, \omega\}$, $v \in \{\Delta+1, \ldots, \Delta+m\}$,
and $v \leq u$.  The shaded region is the sample space of $H_*$ and is
denoted by trapezoid $A$.  Since $X$ and $Y$ are discrete, all of the
probability is contained in masses on the discrete points within the
shaded region.  If we assume that
$\Pr(X=\Delta+1)$, $\Pr(Y=\Delta+1)$, $\Pr(X=\omega)$, and $\Pr(Y=\Delta+m)$ 
are strictly positive, then the edges of $A$ are identifiable.}
\label{fig:trapezoid}
\end{figure}
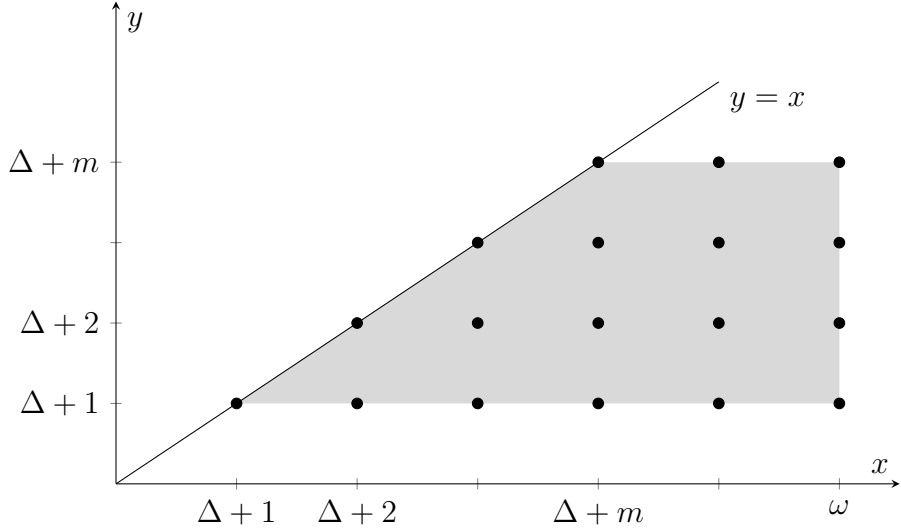

Let
\begin{equation*}
f(u)=\Pr(X=u),\quad g(v)=\Pr(Y=v),\quad \mbox{and} \quad \alpha=\Pr(Y\leq X).
\end{equation*}
The bivariate distribution function $H_*$ over the trapezoid $A$ has probability
mass function (pmf)
\begin{align} \label{eq:h(u,v)_extended}
\begin{split}
    h_*(u,v) & =\Pr(X = u, Y = v \mid Y \leq X)\\
           & =\frac{f(u)g(v)}{\alpha}.
\end{split}
\end{align}
This simple observation tells us that not every distribution over $A$
can be a result of our left-truncation procedure. The marginal distributions
of $H_*$ are given by
\begin{equation*}
    f_*(u) = \Pr(X = u \mid Y \leq X)=\sum_v h_*(u,v),
\end{equation*}
and
\begin{equation*}
  g_*(v) = \Pr(Y = v \mid Y \leq X)=\sum_u h_*(u,v).
\end{equation*}

For our forthcoming results to be meaningful, it must be
possible to express the pmf $f$ (or $g$) in terms of the pmf $h_*$. Notably,
\citet{woodroofe_1985} shows us by his Theorem 1 that we can indeed
do so by expressing the cumulative hazard rate function in terms of
joint cumulative distribution function (cdf) $H_*$.
Since we deal only with discrete random variables, however,
it is more convenient to work with the hazard rate for $X$,
\begin{equation*}
    \lambda(x) = \frac{\Pr(X=x)}{\Pr(X \geq x)},
\end{equation*}
where $x \in \{\Delta+1, \Delta+2,\dots, \omega\}$.
One can show that
\begin{equation}\label{eq:haz_rate_form}
    \lambda(x) = \frac{f_*(x)}{C(x)},
\end{equation}
where
\begin{equation}
    C(x) = \Pr(Y \leq x \leq X \mid Y \leq X)=\sum_{v\leq x\leq u}h_*(u,v).
    \label{eq:C(z)}
\end{equation}

Indeed, first observe that
\begin{align*}
 C(x) &= \Pr(Y \leq x \leq X \mid Y \leq X)\\	  
      &= \frac{1}{\alpha}(\Pr(Y \leq x)-\Pr(X < x, Y \leq x))\\
      &= \frac{1}{\alpha}\Pr(Y \leq x)\Pr(X \geq x).
\end{align*}
Hence,
\begin{equation}
\lambda(x) = \frac{\Pr(X=x)}{\Pr(X \geq x)}
           = \frac{\Pr(X=x, Y\leq X)}{\Pr(Y\leq X)}
             \frac{\Pr(Y\leq X)}{\Pr(X \geq x)\Pr(Y\leq x)}
           = \frac{f_*(x)}{C(x)}.
\label{eq:lam_fu}
\end{equation}
Having $C(x)$ in the denominator is not a concern, because for any $x$,
\begin{equation*}
C(x)\geq h_*(\omega,\Delta+1)=\frac{f(\omega)g(\Delta+1)}{\alpha}>0.
\end{equation*}

The re-construction of the cdf $F$ from the hazard rate $\lambda$ is based on
the following standard result of survival analysis.  For any integer $x$ such
that $\Delta +1 < x \leq \omega$,
\begin{align}
\prod_{\Delta+1\leq k < x} [1 - \lambda(k)]
&= \bigg[\frac{\Pr(X \geq \Delta+2)}{\Pr(X \geq \Delta+1)} \bigg]
  \bigg[\frac{\Pr(X \geq \Delta+3)}{\Pr(X \geq \Delta+2)} \bigg] \cdots
  \bigg[\frac{\Pr(X  \geq x)}{\Pr(X \geq x-1)} \bigg] \nonumber\\
&= \Pr(X \geq x),
\label{eq:F_haz}
\end{align}
with the convention that \eqref{eq:F_haz} is unity for $x \leq \Delta + 1$.
Since $X$ is discrete, it is enough to know $F$ at the jump points.

In a similar fashion, one can derive an analog of
formula~\eqref{eq:haz_rate_form} for what is sometimes known as the
reverse hazard rate function \citep[for a nice introduction, see][]{block_1998}.
The reverse hazard rate is effectively analogous
to the hazard rate in \eqref{eq:haz_rate_form} but backwards-looking.
That is, the reverse hazard rate is the probability of the event of interest
occurring in the current interval, given we know the event of interest
occurred prior to the current interval.  Formally, the reverse hazard rate
is defined as
\begin{equation}
    \beta(y) = \frac{ \Pr(Y = y) }{ \Pr(Y \leq y) }=\frac{g_*(y)}{C(y)},
    \label{eq:rev_haz}
\end{equation}
where $y \in \{\Delta+1, \Delta+2, \dots, \Delta+m\}$.
As a consequence, we get the following formula for the cdf $G$,
\begin{equation}
  \Pr(Y\leq y)= \prod_{\Delta+m \geq k > y} [1 - \beta(k)],
  \label{eq:F_rev_haz}
\end{equation}
where $\Delta+1\leq y \leq \Delta+m$.

\subsection{Estimators}
\label{sec:est_3}

There are different ways to think about sampling in the case of 
left-truncation. For example,
\citet{woodroofe_1985} assumes that there is a population of $X$s and $Y$s,
from which we take a sample of size $N$. Then we apply left-truncation to the
sample, and this gives a sample of left-truncated pairs of {\it random} sample size
$n$. Our thinking, however, is different. We assume that there is the original
population of $X$s and $Y$s. We apply left-truncation to the entire population 
to get a population of left-truncated pairs. Then we extract a sample of
{\it deterministic} size $n$ from the left-truncated population. That is, our
observations are directly from the distribution $H_*$.
Given the practicalities of the securitization process, sampling from
$H_*$ directly is more appropriate for our application than the assumed
sampling process of \citet{woodroofe_1985}.  Phrased differently, our sampling
process effectively samples from the already left-truncated lease data within
the trust rather than imagines we are able to sit with the ABS issuer and see
loans that did not meet the minimum survival requirements to be included in
the trust.  A theoretical divergence with generally limited practical
significance, but its importance is evident with ABS data.

Formally, let $\{(X_i,Y_i)\}_{1\leq i\leq n}$ be i.i.d.\  pairs of random
variables with distribution $H_*$
(i.e., a sample from distribution $H_*$ on trapezoid $A$ of
Figure~\ref{fig:trapezoid}).  This is materially different than the sampling
space of all possible target population pairs of $(X,Y)$ absent
left-truncation.
In other words, referring again to Figure~\ref{fig:trapezoid}, there is a bias
from the left-truncation condition in that some pairs, such as $u = \Delta + 1$
and $v = \Delta + m$, are not observable. This distinction warrants emphasis
because it is erroneous to assume both pairs $(X,Y)$ and $(X_i, Y_i)$, 
$1 \leq i \leq n$, share the same properties (e.g., while $X$ and $Y$ are 
assumed to be independent, $X_i$ and $Y_i$ clearly are not).

We desire to provide interval estimates for the hazard rates of $F_0$, the
reverse hazard rates of $G$, and the cdfs $F_0$ and $G$ from the i.i.d.\  sample
$\{(X_i,Y_i)\}_{1\leq i\leq n}$ with distribution $H_*$.
Examination of~\eqref{eq:haz_rate_form} tells us that the hazard rate
$\lambda(x)$ is a ratio of two probabilities of some events related to random
variables $(X_i, Y_i)$, $1 \leq i \leq n$, 
and we have natural estimates of each probability
within the ratio vis-\`{a}-vis the observed frequencies. 
This suggests the following estimator for the hazard rate,
\begin{equation}
    \hat{\lambda}_n(x) = \frac{\frac{1}{n} \sum_{i=1}^{n}
    \mathbf{1}_{X_i = x}}{\hat{C}_n(x)}, \label{eq:lam_hat}
\end{equation}
where $\mathbf{1}_{(\cdot)}$ is the standard indicator function taking value
$1$ if statement $(\cdot)$ is true and 0 otherwise, and
\begin{equation}
    \hat{C}_n(x) = \frac{1}{n} \sum_{j=1}^{n} \mathbf{1}_{Y_j \leq x \leq X_j}.
    \label{eq:C_n(x)}
\end{equation}
By employing the same method of~\eqref{eq:F_haz}, we immediately get an 
estimator for the cdf~$F_0$,
\begin{equation}
  \hat{F}_n(x) = 1 -
  \prod_{\Delta+1 \leq k \leq x} [1 - \hat{\lambda}_n(k)].
  \label{eq:F_est}
\end{equation}
In similar fashion we can produce the following estimator of the 
reverse hazard rate
$\beta(y)$ of $Y$,
\begin{equation}
  \hat{\beta}_n(y) = \frac{\frac{1}{n} \sum_{i=1}^{n} \mathbf{1}_{Y_i = y}}{\hat{C}_n(y)},
  \label{eq:bet_hat}
\end{equation}
and the cdf $G$,
\begin{equation}
  \hat{G}_n(y) = \prod_{\Delta+m \geq k > y} [1 -\hat{\beta}_n(k)].
  \label{eq:G_est}
\end{equation}

It is theoretically satisfying that the estimators \eqref{eq:F_est} and
\eqref{eq:G_est} coincide with the corresponding estimators (8) and (9) in
\citet{woodroofe_1985}, despite the alternative constructive path our
discrete data framework required.  It is in proceeding to analyze the
asymptotic properties of the
estimators \eqref{eq:lam_hat} and \eqref{eq:bet_hat}, however, that we can
take advantage of the discrete structure of $h_*$ in 
\eqref{eq:h(u,v)_extended} to directly address complications traditionally
assumed away.  For example, we can handle ties among the discrete sample
space of $X$ and $Y$ in proving the vector of estimates \eqref{eq:lam_hat}
and \eqref{eq:bet_hat} are together the maximum likelihood estimate (MLE) of
the parameters of the conditional bivariate distribution $H_*$
(\citet{woodroofe_1985}
assumes no ties, for example).  Furthermore, in deriving the estimator's
asymptotic properties, many authors (see Appendix~\ref{sec:lit_review}) assume
continuous $F$ and $G$ to avoid convergence argument complications
introduced by potentially unexpected discrete point masses in the distribution
functions (for an example of the complications in trying to account for such
point masses without a priori knowledge, see the \textit{change point} analysis
and proofs of \citet{rabhi_2017}).

As alluded to in the previous paragraph, it is noteworthy that the joint vector
of estimates with components \eqref{eq:lam_hat} and \eqref{eq:bet_hat}
can be shown to be the MLE of the parameters of the discrete
conditional bivariate distribution $H_*$.  That $H_*$ is a parametric
distribution may not be obvious.  To see this, observe that $h_*(u,v)$ defined
in \eqref{eq:h(u,v)_extended} is a function of the discrete mass probabilities
$0 \leq f(u) \leq 1$, $\Delta + 1 \leq u \leq \omega$, and
$0 \leq g(v) \leq 1$, $\Delta + 1 \leq u \leq \Delta + m$.
Thus, the probabilities $f$ and $g$ are
in actuality parameters (of which only $m + \omega - 2$ are free, as we require
$\sum_{u} f(u) = \sum_{v} g(v) = 1$).

There also exists an equivalent one-to-one parameterization of $h_*(u,v)$ using
the hazard rates $\lambda$ and $\beta$.  Specifically, from \eqref{eq:lam_fu}
and \eqref{eq:F_haz},
\begin{equation}
f(u) = \lambda(u) \prod_{k = 1}^{u - 1} [1 - \lambda(k)],
\quad \text{and} \quad
\lambda(u) = \frac{ f(u) }{1 - \sum_{k=1}^{u-1} f(k)}, \quad
\Delta + 1 \leq u \leq \omega,
\label{eq:f_to_lam}
\end{equation}
with the conventions $\prod_{k = 1}^{0} [1 - \lambda(k)] = 1$ and
$\sum_{k=1}^{0}f(k) = 0$. Similarly, from
\eqref{eq:rev_haz} and \eqref{eq:F_rev_haz},
\begin{equation}
g(v) = \beta(v) \prod_{k = v+1}^{\Delta + m} [1 - \beta(k)],
\quad \text{and} \quad
\beta(v) = \frac{ g(v) }{1 - \sum_{k=v+1}^{\Delta + m} g(k) }, \quad
\Delta + 1 \leq v \leq \Delta + m,
\label{eq:g_to_bet}
\end{equation}
with the conventions
$\prod_{k = \Delta + m + 1}^{\Delta + m} [1 - \beta(k)] = 1$ and
$\sum_{k = \Delta + m + 1}^{\Delta + m} g(k) = 0$.  That there are still
$m + \omega - 2$ free parameters is evident from the known hazard rate
probabilities
$\lambda(\omega) = \beta(\Delta + 1) = 1$.  With this background, we formally
state the MLE property of the joint vector of estimates with
components~\eqref{eq:lam_hat} and~\eqref{eq:bet_hat} in terms of the parameters
of $H_*$ in Theorem~\ref{thm:MLE}.  The complete proof may be found in
Appendix~\ref{sec:proofs:mle}, and we also present an outline of the proof
immediately following Theorem~\ref{thm:MLE}, as it may be of interest
to some readers.

\begin{theorem}
Define the discrete-mass trapezoid
\begin{equation*}
    \mathcal{A} = \{ (u,v) \in \mathbb{N} : \Delta + 1 \leq u \leq \omega,
    \Delta + 1 \leq v \leq \Delta + m, v \leq u \},
\end{equation*}
and consider the bivariate distribution $h_*(u,v)$ defined in \eqref{eq:h(u,v)_extended}
over $\mathcal{A}$.  
Let $\{(X_i, Y_i)\}_{1 \leq i \leq n}$ be $n$ independent and identically distributed
pairs of observations sampled from $h_*(u,v)$ such that 
$\hat{f}_{*,n}(u) = \frac{1}{n} \sum_{i=1}^{n} \mathbf{1}_{X_i = u} > 0$ 
for $u \in \mathcal{A}$, 
and $\hat{g}_{*,n}(v) = \frac{1}{n} \sum_{i=1}^{n} \mathbf{1}_{Y_i = v} > 0$ 
for $v \in \mathcal{A}$.
Further define 
\begin{equation}
    \hat{\bm{\Lambda}}_n = 
    (\hat{\lambda}_n(\Delta+1), \ldots, \hat{\lambda}_n(\omega-1),1)^{\top},
    \label{eq:lam_vec}
\end{equation}
and
\begin{equation}
    \hat{\mathbf{B}}_n = (1, \hat{\beta}_n(\Delta+2), \ldots,
    \hat{\beta}_n(\Delta+m))^{\top},
    \label{eq:bet_vec}
\end{equation}
where $\hat{\lambda}_n$ and $\hat{\beta}_n$ follow from \eqref{eq:lam_hat} and
\eqref{eq:bet_hat}, respectively.
Then the joint vector of estimates,
$(\hat{\bm{\Lambda}}_n, \hat{\mathbf{B}}_n)^{\top}$,
is a MLE for the parameters, $\lambda(u)$, $\beta(v)$,
$u,v \in \mathcal{A}$, of the bivariate distribution
$h_*(u,v)$.
\label{thm:MLE}
\end{theorem}

\begin{proof}[Proof Outline]
From the
one-to-one correspondence of the two parameterizations of the distribution
$H_*$ and the invariance property of the MLE
\citep[e.g.,][Theorem~7.2.1, pg.~350]{nitis_2000}, it is sufficient to find the
MLE for the parameters $f$ and $g$ and demonstrate they are exactly equal to
the estimates \eqref{eq:lam_hat} and \eqref{eq:bet_hat} in the same
form as \eqref{eq:f_to_lam} and \eqref{eq:g_to_bet}, respectively.  It is
preferable to define the likelihood in terms of the parameters
$f$ and $g$ because the equivalent likelihood with a parameterization in terms
of $\lambda$ and $\beta$ is cumbersome.

To maximize the likelihood, we restrict
the parameter space of $f$ and $g$ to the convex set of all $0 < f, g < 1$ such
that $\sum_{u} f(u) = \sum_{v} g(v) = 1$.
The convexity of this restricted parameter space in conjunction with the
behavior of the loglikelihood on the boundary of the parameter space confirms
that the maximum point must lie within the restricted parameter space. 
We next solve the system of partial derivatives with
respect to each parameter $f(u)$, $g(v)$ $u, v \in \mathcal{A}$, equated to
zero sequentially to show the solution set admits only one solution, which must
therefore be the global maximum and MLE.  If we move sequentially from the
minimum points of $u$ and $v$, we can show the MLE for each $f$ is exactly of
the form \eqref{eq:f_to_lam}. The complete result for $g$ then follows by
solving the system of partial derivative equations equated to zero
sequentially from the maximum points of $u$ and $v$ (i.e., symmetry).

\end{proof}

There are some related results.  For example, \citet{vardi_1982} finds the
MLE in the situation of a \textit{length-biased distribution}, but the sampling
procedure is not from $h_*$.  Instead, two independent samples are used, the
latter of which is from a length-biased distribution.  In the derivation of
\citet{woodroofe_1985}, the sampling procedure used therein is not from $h_*$
but instead the complete pairs $(X,Y)$, of which a random quantity are truncated
(i.e., whenever $Y > X$).  Further, the results are given assuming no ties among
the left-truncation or lifetime distribution observations.  \citet{wang_1987}
also assumes the same sampling procedure as \citet{woodroofe_1985} (i.e.,
stopping time theory).  \citet{keiding_1990} assume throughout that 
$\Pr(Y = X) = 0$, which is also not necessary in our framework.  Furthermore,
as indicated in Appendix~\ref{sec:lit_review}, the discrete case is largely
left unstudied.  We thus find our proof of Theorem~\ref{thm:MLE} to be the
first direct proof in the literature that the estimation vector with
components \eqref{eq:lam_hat} and \eqref{eq:bet_hat} is indeed the MLE for
the parameters of the discrete conditional distribution $H_*$.

\section{Asymptotic Results}
\label{sec:asym_res}

We now establish asymptotic normality of the hazard rate and reverse hazard rate
estimators, along with the unanticipated result of independence.  We also
prove asymptotic normality of the survival function estimator for $X$ and
distribution function estimator for $Y$.  Finally, this section closes with
a hypothesis test for the shape of $G$.  All corresponding proofs may be found
in Appendix~\ref{sec:proofs}.  We set the stage with three comments as follows.

First, notice that throughout this section, as before, $X$ and $Y$ are positive
discrete random variables with distribution functions $F$ and $G$,
respectively, and $\{(X_i, Y_i)\}_{1 \leq i \leq n}$ are independent and
i.i.d.\ distributed pairs of random variables with distribution $H_*$.
More specifically,  $\{(X_i, Y_i)\}_{1 \leq i \leq n}$ are a random sample
from a population with probability mass function $h_*$
in Equation~\eqref{eq:h(u,v)_extended},  spanning the finite set of points on 
the plane with integer-valued coordinates $(u,v)$ such that 
$u \in \{\Delta + 1, \ldots, \omega\}$,
$v \in \{\Delta + 1, \ldots, \Delta + m\}$, 
$\Delta + m \leq \omega$, and $v \leq u$ (trapezoid $A$
of Figure~\ref{fig:trapezoid}). Additionally, we will
continue to assume that $\Pr(X=\Delta+1)$, $\Pr(Y=\Delta+1)$, $\Pr(X=\omega)$,
and $\Pr(Y=\Delta+m)$ are strictly positive. See Figure~\ref{fig:trapezoid} as
necessary.

Second, recall that we apply left-truncation to the entire population of
$X$ and $Y$, which yields a population of left-truncated pairs.  From this
left-truncated population, we draw a sample of deterministic size $n$.  
Therefore, in what follows, we investigate the  limiting behavior as
$n \rightarrow \infty$.

Third and finally, to state our asymptotic results it is convenient to
introduce the following notation. Let $(X_i, Y_i)$, $1 \leq i \leq n$ be a
sample pair from $h_*$.  Then we denote
\begin{align}
    c(u,v) &= \Pr(Y_i \leq u \leq X_i, Y_i \leq v \leq X_i) \nonumber\\
    &= \Pr(X \geq \max(u,v), Y \leq \min(u,v) \mid Y \leq X) \nonumber\\
    &= \sum_{y=\Delta+1}^{\min(u,v)} \sum_{x=\max(u,v)}^{L} h_*(x,y)
    \nonumber\\
    &= \frac{1}{\alpha} \Pr(Y \leq \min(u,v)) \Pr(X \geq \max(u,v)).
    \label{eq:c(u,v)}
\end{align}
The various equation steps have been left as a form of summary of probability
statements to show the connection between a single sampled pair $(X_i, Y_i)$,
$1 \leq i \leq n$, from $h_*$, the pair $(X,Y)$ conditional on $Y \leq X$,
the pmf $h_*$ itself (i.e., Figure~\ref{fig:trapezoid}), and the importance of
assuming independence between $X$ and $Y$.
Finally, we observe $c(z,z) = C(z)$ and $c(u,v) = c(v,u)$.  

\subsection{Hazard and Reverse Hazard Rate}

We first inspect the estimator $\hat{C}_n$ in the denominator of the
Woodroofe-type estimator~\eqref{eq:lam_hat} with the multivariate 
Central Limit Theorem (CLT).

\begin{Lemma}[$\hat{\mathbf{C}}_n$ Asymptotic Normality]
\label{thm:Cn}
Define
$\hat{\mathbf{C}}_n = (\hat{C}_n(\Delta + 1), \ldots, \hat{C}_n(\omega))^{\top}$. Then,
\begin{equation*}
    \sqrt{n} ( \hat{\mathbf{C}}_n - \mathbf{C} )
    \overset{\mathcal{L}}{\longrightarrow} N( \mathbf{0}, \mathbf{\Sigma}_c),
    \text{ as } n \rightarrow \infty,
\end{equation*}
where $\mathbf{C} = (C(\Delta + 1), \ldots, C( \omega ))^{\top}$ and 
$\bm{\Sigma}_c$ is a covariance matrix $\lVert \sigma_{k',k} \rVert$ such that
\begin{equation*}
	\sigma_{k',k} = \begin{cases} C(k)[1 - C(k)], & k' = k\\ c(k',k) - C(k')C(k),
	& k' \neq k \end{cases},
\end{equation*}
for $k', k = \Delta + 1, \Delta + 2, \ldots, \omega$.
\end{Lemma}

\begin{Lemma}
\label{lem:Cn}
As $n \rightarrow \infty$,
$\hat{\mathbf{C}}_n \overset{\mathcal{P}}{\longrightarrow} \mathbf{C}$.
\end{Lemma}

The discrete nature of $X$ and $Y$ along with the finite sample space
of trapezoid $A$ yields attractive mathematical conveniences, which lead
themselves naturally to computational programming.  The same is true
for the Woodroofe-type estimator of the hazard rate $\lambda$, which we now
examine.

\begin{theorem}[$\hat{\bm{\Lambda}}_n$ Asymptotic Normality]
\label{thm:haz}
For $\hat{\bm{\Lambda}}_n$ defined in \eqref{eq:lam_vec},
\begin{equation*}
    \sqrt{n}( \hat{\bm{\Lambda}}_n - \bm{\Lambda} ) \overset{\mathcal{L}}
    {\longrightarrow} N \big(\bm{0}, \bm{\Sigma}_f \big),
    \text{ as } n \rightarrow \infty,
\end{equation*}
where $\bm{\Lambda} = (\lambda(\Delta + 1), \lambda(\Delta+2), \ldots,
\lambda(\omega))^{\top}$ with $\lambda(z) = f_*(z) / C(z)$ and
\begin{equation}
    \bm{\Sigma}_f = \textup{diag} \bigg(
    \frac{f_*(\Delta+1)c(\Delta+1, \Delta+2)}{C(\Delta+1)^3}, \ldots,
    \frac{f_*(\omega-1)c(\omega-1,\omega)}{C(\omega-1)^3}, 0 \bigg).
    \label{eq:thm3_sigma}
\end{equation}
That is, the estimators $\hat{\lambda}_n(\Delta + 1)$, $\ldots$,
$\hat{\lambda}_n(\omega)$ are asymptotically normal and independent.
\end{theorem}

\begin{remark}
There is an alternative form of $\bm{\Sigma}_f$ that may be preferable.
Observe for $x \in \{\Delta+1, \ldots, \omega\}$,
\begin{align*}
	\frac{ f_*(x) c(x, x+1) }{ C(x)^3 }
	&= \frac{ f_*(x) }{ C(x) } \frac{ \alpha^{-1} \Pr(Y \leq x) \Pr(X \geq x+1)}
	{ [\alpha^{-1} \Pr(Y \leq x) \Pr(X \geq x)]^2 } \\
	&= \frac{ \lambda(x)^2 [1 - \lambda(x)] }{ f_*(x) }. 
\end{align*}
Hence, alternatively,
\begin{equation}
\bm{\Sigma}_f = \textup{diag} \bigg(
    \frac{\lambda(\Delta+1)^2 [1 - \lambda(\Delta+1)]}{f_*(\Delta+1)}, \ldots,
    \frac{\lambda(\omega-1)^2 [1 - \lambda(\omega-1)]}{f_*(\omega-1)}, 0 \bigg).
    \label{eq:thm3_sigma_simp}
\end{equation}
\end{remark}

Further, \eqref{eq:thm3_sigma} and~\eqref{eq:thm3_sigma_simp} are equivalent
when the true quantities are replaced by their MLEs.
That is, for $x \in \{\Delta+1, \ldots, \omega\}$ with
\begin{equation}
\hat{f}_{*,n}(x) = \frac{1}{n} \sum_{i=1}^{n} \mathbf{1}_{X_i = x},
\quad \mbox{and} \quad
\hat{c}_n(x,x+1) = \frac{1}{n} \sum_{i=1}^{n}
\mathbf{1}_{Y_i \leq x,
X_i \geq x+1},
\label{eq:fs_emp}
\end{equation}
it is easy to show
\begin{equation*}
\frac{ \hat{f}_{*,n}(x) \hat{c}_n(x,x+1) }{ \hat{C}_n(x)^3 } =
\frac{ \hat{\lambda}_n(x)^2 [1 - \hat{\lambda}_n(x)] }{ \hat{f}_{*,n}(x) }.
\end{equation*}

We now state the following corollary without proof for completeness.

\begin{corollary}
\label{cor:Hn}
As $n \rightarrow \infty$,
$\hat{\bm{\Lambda}}_n \overset{\mathcal{P}}{\longrightarrow} \bm{\Lambda}$.
\end{corollary}

When estimating $G$ is of interest, we may also obtain the sister
statement for the reverse hazard rate $\beta$ as follows.

\begin{theorem}[$\hat{\mathbf{B}}_n$ Asymptotic Normality]
\label{thm:button}
For $\hat{\mathbf{B}}_n$ defined in \eqref{eq:bet_vec},
\begin{equation*}
    \sqrt{n}( \hat{\mathbf{B}}_n - \mathbf{B} ) \overset{\mathcal{L}}{
    \longrightarrow} N \big(\bm{0}, \bm{\Sigma}_g \big),
    \text{ as } n \rightarrow \infty,
\end{equation*}
where $\mathbf{B} = (\beta(\Delta + 1), \beta(\Delta+2), \ldots, 
\beta(\Delta+m))^{\top}$ with $\beta(z) = g_*(z) / C(z)$ and
\begin{equation*}
    \bm{\Sigma}_g = \textup{diag} \bigg(0, \frac{g_*(\Delta+2)c(\Delta+1,
    \Delta+2)}{C(\Delta+2)^3}, \ldots, \frac{g_*(\Delta+m)c(\Delta+m-1,
    \Delta + m)}{C(\Delta+m)^3}\bigg).
\end{equation*}
That is, the estimators $\hat{\beta}_n(\Delta + 1), \ldots, \hat{\beta}_n(
\Delta + m)$ are asymptotically normal and independent.
\end{theorem}

\begin{remark}
One may also write
\begin{equation}
\bm{\Sigma}_g = \textup{diag} \bigg(0, \frac{\beta(\Delta+2)^2
[1 - \beta(\Delta+2)]}{g_*(\Delta+2)}, \ldots,
\frac{\beta(\Delta+m)^2 [1 - \beta(\Delta+m)]}{g_*(\Delta+m)}
\bigg). \label{eq:sig_g_2}
\end{equation}
\end{remark}

We again state the following corollary without proof for completeness.

\begin{corollary}
As $n \rightarrow \infty$,
$\hat{\mathbf{B}}_n \overset{\mathcal{P}}{\longrightarrow} \mathbf{B}$.
\end{corollary}

\subsection{Survival and Distribution Function}

For most analysts of survival data, the key quantity of interest is the
survival function, $S(x) = 1 - F(x)$.  From~\eqref{eq:lam_hat} and
\eqref{eq:F_est}, we have the estimator
\begin{equation}
    \hat{S}_n(x) = \prod_{\Delta+1 \leq k \leq x} [1 - \hat{\lambda}_n(k)].
    \label{eq:s_hat}
\end{equation}
Asymptotic normality also extends to~\eqref{eq:s_hat}, which we now show.

\begin{theorem}[$\hat{\mathbf{S}}_n$ Asymptotic Normality]
\label{thm:Sn}
For the estimator
$\hat{\mathbf{S}}_n = (\hat{S}_n(\Delta + 1), \hat{S}_n(\Delta+2), \ldots,
 \hat{S}_n(\omega))^{\top}$,
\begin{equation*}
    \sqrt{n}( \hat{\mathbf{S}}_n - \mathbf{S}) \overset{\mathcal{L}}{
    \longrightarrow} N \big(0, \mathbf{RK}\bm{\Sigma}_f[\mathbf{RK}]^{\top}
    \big), \text{ as } n \rightarrow \infty,
\end{equation*}
where 
$\mathbf{S} = (S(\Delta+1), S(\Delta+2), \ldots, S(\omega))^{\top}$, 
$\bm{\Sigma}_f$ follows from Theorem~\ref{thm:haz},
\begin{equation*}
\mathbf{K} = \begin{bmatrix} - [1-\lambda(\Delta+1)]^{-1} & 0 & \ldots & 0 \\
 - [1-\lambda(\Delta+1)]^{-1} & - [1-\lambda(\Delta+2)]^{-1} & \ldots & 0\\
\vdots & \vdots & \ddots & \vdots\\
- [1-\lambda(\Delta+1)]^{-1} & - [1-\lambda(\Delta+2)]^{-1} & \ldots & -
[1-\lambda(\omega)]^{-1}
\end{bmatrix},
\end{equation*}
and $\mathbf{R} = \textup{diag}(S(\Delta+1), S(\Delta+2), \ldots, S(\omega))$.
\end{theorem}

The sister theorem to estimator $G$ is as follows.

\begin{theorem}[$\hat{\mathbf{G}}_n$ Asymptotic Normality]
\label{thm:Gn}
For the estimator
$\hat{\mathbf{G}}_n = (\hat{G}_n(\Delta +1), \hat{G}_n(\Delta+2), \ldots,
 \hat{G}_n(\Delta + m))^{\top}$
\begin{equation*}
    \sqrt{n}( \hat{\mathbf{G}}_n - \mathbf{G}) \overset{\mathcal{L}}{
    \longrightarrow} N \big(0, \mathbf{WM}\bm{\Sigma}_g[\mathbf{WM}]^{\top} \big),
    \text{ as } n \rightarrow \infty,
\end{equation*}
where $\mathbf{G} = (G(\Delta+1), G(\Delta+2), \ldots, G(\Delta+m))^{\top}$,
$\bm{\Sigma}_g$ follows from Theorem~\ref{thm:button},
\begin{equation*}
\mathbf{M} = \begin{bmatrix} - [1-\beta(\Delta+1)]^{-1} & - [1-\beta(
\Delta+2)]^{-1} & \ldots & - [1-\beta(\Delta + m)]^{-1}\\
0 & -[1-\beta(\Delta+2)]^{-1} & \ldots & -[1-\beta(\Delta+m)]^{-1}\\
\vdots & \vdots & \ddots & \vdots\\
0 & 0 & \ldots & -[1-\beta(\Delta+m)]^{-1}
\end{bmatrix},
\end{equation*}
and $\mathbf{W} = \textup{diag}(G(\Delta+1), G(\Delta+2), \ldots, G(
\Delta+m))$.
\end{theorem}

\subsection{Hypothesis Test}

In many applications, it is desirable to test if the distribution of $F$ or
$G$ corresponds to a known distribution.  As one may anticipate, there is
some history in the literature.  For a starting point, we encourage the
reader to review \citet{hyde_1977} and the associated citations.  For our
purposes, we review a few notable examples.  (For consistent terminology
with the associated references within this paragraph,
we will leave the form of truncation and censoring
unspecified; e.g., simply ``truncation" rather than left-truncation;
"censored" rather than right-censored).
To begin, \citet{guilbaud_1988}
generalizes the ordinary Kolmogorov--Smirnov one-sample tests based on the
product-limit estimator.  The test we develop from Theorem~\ref{thm:hypo_gen}
is more akin to a goodness-of-fit test, however.  \citet{mandel_2007} is
related, though they assume continuous $F$ and $G$ to introduce several
goodness-of-fit tests for the truncation distribution.  Similarly,
\citet{hwang_2008} assume the lifetime, truncation, and censoring random
variables are continuous in proposing a chi-square test to test the hypothesis
that the truncation distribution follows a parametric family.  Further, the
asymptotic properties of the nonparametric test of \citet{ning_2010} were
derived assuming a continuous survival function.  See also
\citet{moreira_2014}, in which goodness-of-fit tests are proposed for a
semiparametric model under random double truncation.  As there is no clear
application to discrete $F$ or discrete $G$, we extend our results to propose
a hypothesis testing procedure using a chi-square random variable.
We state our results for the left-truncation distribution $G$.

\begin{theorem}
\label{thm:hypo_gen}
Assume that $G$
follows a known distribution over the discrete points 
$\{\Delta + 1, \ldots, \Delta + m\}$.  Then the test statistic
\begin{equation*}
\mathbb{Q}_G = [\sqrt{n} (\hat{\mathbf{B}}_n^* - \mathbf{B}^*)]^{\top}
[\bm{\Sigma}g^*]^{-1}
[\sqrt{n} (\hat{\mathbf{B}}_n^* - \mathbf{B}^*)]
\overset{\mathcal{L}}{\longrightarrow} \chi^2_{q},
\end{equation*}
where $\hat{\mathbf{B}}_n^* = \big(\hat{\beta}_n(\Delta+2), \ldots,
\hat{\beta}_n (\Delta+m)\big)^{\top}$, 
$\mathbf{B}^* = \big(\beta(\Delta+2), \ldots, \beta(\Delta+m) \big)^{\top}$,
\begin{equation*}
\bm{\Sigma}_g^* = \textup{diag} \bigg(\frac{\beta(\Delta+2)^2
[1 - \beta(\Delta+2)]}{g_*(\Delta+2)}, \ldots,
\frac{\beta(\Delta+m)^2 [1 - \beta(\Delta+m)]}{g_*(\Delta+m)}
\bigg),
\end{equation*}
and $q = \mathbf{card}\{\Delta+2, \ldots, \Delta + m\}$.  The point 
$\Delta + 1$ with the degenerate estimator
$\hat{\beta}_n(\Delta+1)=1$ and $Var[\hat{\beta}_n(\Delta+1)]=0$ is omitted
from $\mathbb{Q}_G$.
\end{theorem}

Specifically, it is often of interest to test if $G$ follows a uniform
distribution. This is an important assumption in the case length-biased
sampling, see for instance \citet{asgharian_2002} and
\citet{una_alvarez_2004}.  There is a long history of proposed methods
for checking this assumption in the literature.
For example, one stated use of the NPMLE
for the left-truncated distribution derived by \citet{wang_1991} is to informally
check the validity of the stationarity assumption (i.e., $G$ coincides with a
uniform distribution).  Similarly, \citet{asgharian_2006} propose a graphical
method to check the stationarity of the underlying incidence times.
\citet{addona_2006} propose a formal test for stationarity of the incidence
rate, but they require a continuous truncation density (via Theorem 1 of
\citet{asgharian_2006}).  Our test, however, allows for exact $p$-value
calculations and considers discrete $G$.  Formally, Corollary~\ref{cor:hypo}
may be used to test if the
left-truncation random variable follows a discrete uniform distribution.

\begin{corollary}
\label{cor:hypo}
Assuming the conditions and notation of Theorem~\ref{thm:hypo_gen},
under the null hypothesis that $G$ is a discrete uniform distribution over
$\{\Delta + 1, \ldots, \Delta + m\}$, the test statistic
\begin{equation*}
  \mathbb{Q}_U =
  [\sqrt{n} (\hat{\mathbf{B}}_n^*
  - \mathbf{B}_U^*)]^{\top}[\bm{\Sigma}_{g,U}^*]^{-1}
  [\sqrt{n} (\hat{\mathbf{B}}_n^* - \mathbf{B}_U^*)]
  \overset{\mathcal{L}}{\longrightarrow} \chi^2_{q},
\end{equation*}
where $\mathbf{B}^*_U = (1/2, 1/3, \ldots, 1/m)$, and
\begin{equation*}
\bm{\Sigma}_{g,U}^* = \textup{diag} \bigg(\frac{[1/2]^2
[1 - 1/2]}{\hat{g}_{*,n}(\Delta+2)}, \ldots,
\frac{[1/m]^2 [1 - 1/m]}{\hat{g}_{*,n}(\Delta+m)}
\bigg).
\end{equation*}
\end{corollary}

Consequently, for $H_0$ that $Y$ is discrete uniformly distributed and
significance level $0 \leq \alpha \leq 1$, one can reject
$H_0$ if $\mathbb{Q}_U \leq \chi^2_{q, \alpha / 2}$ or
$\mathbb{Q}_U \geq \chi^2_{q, 1-\alpha / 2}$, where
$\chi^2_{q, \theta}$ is
the $(100 \times \theta)$th $(0 < \theta < 1)$ percentile of a chi-square
distribution with $q$ degrees of freedom.
The accuracy of the asymptotic chi-square distribution was
investigated for a discrete uniform~$G$ in our simulation study. The
empirical distribution of the test statistics matches closely to the
limiting chi-square distribution, which we validated down to a sample
size of $n = 500$.

\section{Simulation Study}
\label{sec:sim}

In this section, we examine the finite sample behavior of the estimation
vectors $\hat{\bm{\Lambda}}_n$ and $\hat{\mathbf{B}}_n$.  In
addition to serving as an experimental verification of Theorems~\ref{thm:haz}
and~\ref{thm:button}, our intention is to help practitioners estimate the
minimum sample size of discrete-time left-truncated data needed to achieve
a desired level of estimation accuracy.  We proceed in two parts.  First,
for the purposes of illustration,
we will consider a combination of classical distributions for the 
lifetime and left-truncation random variables.  Second, 
with an eye towards our application, the section will close
with a combination of lifetime and left-truncation random 
variables that is a closer approximation to those observed within structured 
finance (e.g., Section~\ref{sec:app}).

Assume first that $Y$ follows a discrete uniform distribution over
$\mathcal{Y} = \{1,2,\ldots, 10\}$, and that $X$ follows a truncated geometric
distribution over $\mathcal{X} = \{1,2, \ldots,24\}$.  Specifically, the pmf of
$X$ is
\begin{equation}
\Pr(X = x) = \begin{cases} p(1-p)^{x-1}, & x = 1, 2, \ldots, 23,\\
\sum_{x=24}^{\infty} p(1-p)^{x-1}, & x = 24,\\
0, & \text{otherwise},
\end{cases}
\label{eq:X_trunc}
\end{equation}
where $0 < p < 1$. From this, we may calculate many key quantities
of interest.  For example, with $p = 0.20$,
\begin{equation*}
\alpha = \Pr(Y \leq X) = \sum_{y = 1}^{10} \Pr(Y = y)  \Pr(X \geq y)
= 0.4463,
\end{equation*}
as well as the useful quantities~\eqref{eq:C(z)}, \eqref{eq:c(u,v)},
\eqref{eq:r(u,v)}, and~\eqref{eq:s(u,v)}.  Notice here that $\Delta = 0$,
$m = 10$, and $\omega = 24$.

We remark here on the behavior of~\eqref{eq:thm3_sigma} across various values
of $p$. For larger values of $p$, the variance of
$\hat{\lambda}_n$ for
values of $X$ closer to 23 quickly explodes.  This is not unexpected because,
as $p$ increases, it becomes more and more unlikely to observe large values
of $X$.  On the other hand, for very small values of $p$ close to zero, the
variance of $\hat{\lambda}_n$ for values of $X$ close to 1 is the largest
and rapidly decreases until $X = 10$, the final possible left-truncation point.
This suggests what
we can already glean from a careful examination of~\eqref{eq:thm3_sigma_simp}:
estimation accuracy of $\hat{\lambda}_n(x)$ is dependent on the quantities
$\lambda(x)$ and $f_*(x)$.

In addition to demonstrating the asymptotic unbiasedness of the estimators
\eqref{eq:lam_hat} and \eqref{eq:bet_hat}, we will also compare the empirical
covariance against the asymptotic covariance suggested by
Theorems~\ref{thm:haz} and~\ref{thm:button} by examining the resulting
confidence interval estimates.  As is standard practice, we will keep the
intervals within the unit interval by constructing the
confidence interval estimates on a log scale with the delta method
\citep[e.g.,][Theorem 5.3.5, pg. 261]{nitis_2000}
and then
transforming them back exponentially to the original scale.  
For example, the 95\% confidence
intervals for $\lambda(x)$, $x \in \{1, \ldots, 24\}$, are
\begin{equation}
\exp \bigg{ \{ } \ln \hat{\lambda}_n(x) \pm 1.96 \times
\sqrt{ \frac{ 1 - \hat{\lambda}_n(x) }{\hat{f}_{*,n}(x) \times n} } \bigg{ \} }.
\label{eq:CI_true}
\end{equation}


In our analysis summarized in the first two rows of Figure~\ref{fig:CI_comps},
we demonstrate the estimator's asymptotic unbiasedness and normality
(we assume $p = 0.20$ and consider 1{,}000 replicates). 
For the asymptotic unbiasedness, we plot the true hazard and reverse
hazard rates (dashed black lines) against an average of the 1{,}000
estimated replicates (blue lines).  The two are largely indistinguishable,
especially as $n$ increases.  For the asymptotic normality and covariance
structure specified within Theorems~\ref{thm:haz} and~\ref{thm:button},
we compare three quantities. The first quantity is the true 95\% confidence
interval, and it is represented by the blue ribbon.  The second quantity
is the average of the estimated confidence intervals derived from the
simulated data, i.e., \eqref{eq:CI_true}, over the 1,000 replicates. We
represent this quantity by the red ribbon.  The third quantity is the middle
95th empirical percentile of the 1{,}000 replicates, and it is represented
by the purple ribbon.  Once again, all closely agree, especially as $n$
increases.  We also found that the off-diagonal elements in
the empirical covariance calculation across the 1{,}000 replicates
of all estimators are each very close to zero, which is further
experimental validation of independence.
Readers interested in the full empirical covariance matrix
comparison may contact the corresponding author for more details.

\begin{figure}[tbh!]
    \includegraphics[width=\textwidth]{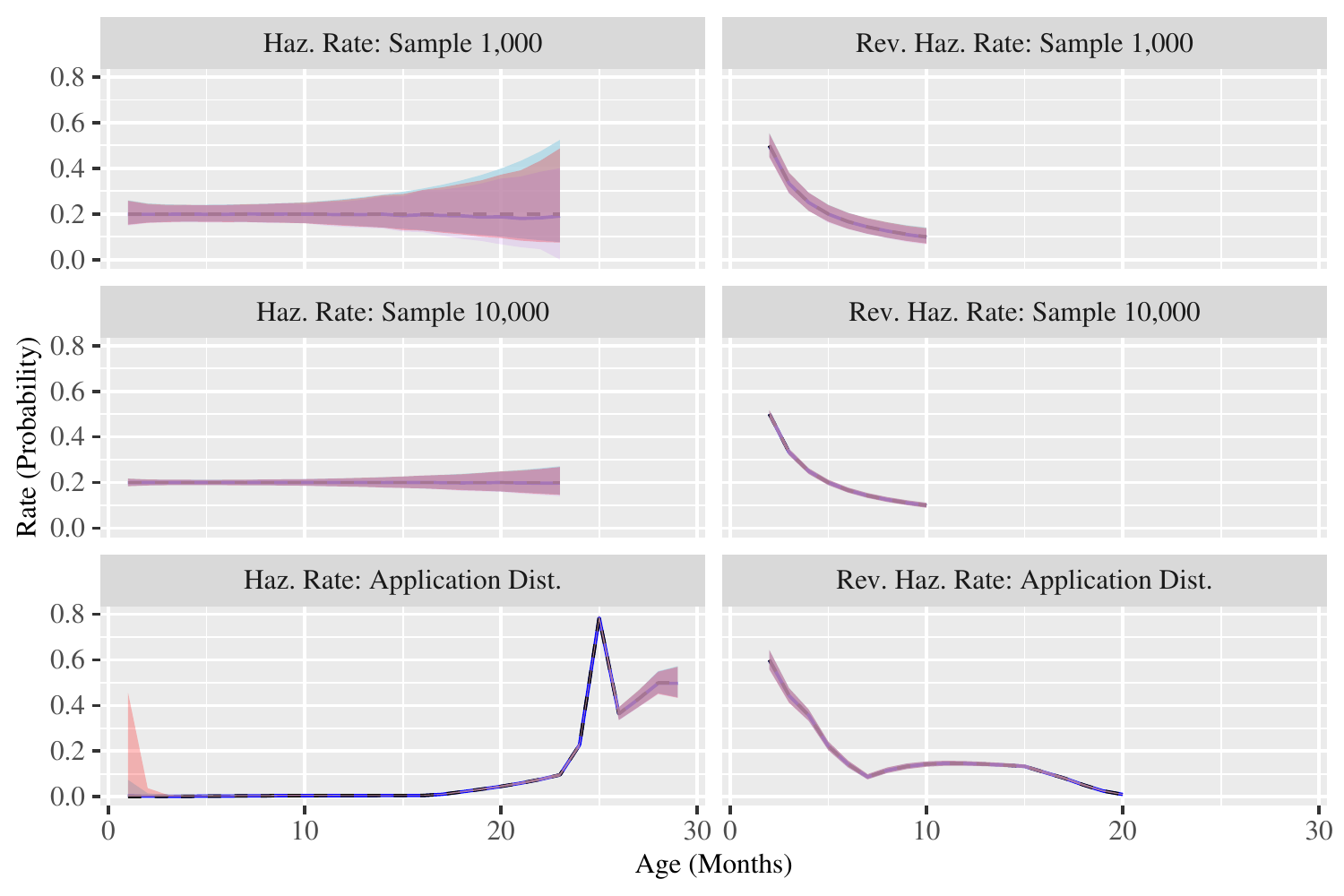}
        \caption{A simulation verification of Theorems~\ref{thm:haz}
        (i.e., $\lambda$, left-column) and~\ref{thm:button} 
		(i.e., $\beta$, right-column)        
        for sample sizes of $n = 1{,}000$ and
        $n = 10{,}000$ with classical distributions
        (the lifetime distribution follows a
        truncated geometric distribution at $x=24$, as defined in
        \eqref{eq:X_trunc} with $p = 0.20$ and the left-truncation
        distribution follows a discrete uniform distribution over the
        integers $\{1, \ldots, 10\}$) and distributions
        more representative of an application to structured
        finance (compare Figures~\ref{fig:appDist} and
        \ref{fig:mbalt}).
        Each chart presents a comparison
        of an average over all replicates of the estimate (blue lines)
        and true (dashed black lines),
        which are largely indistinguishable.  Further the 95\% true
        confidence intervals (blue ribbon), an average over all
        replicates of the 95\% confidence intervals
        estimated from the empirical quantities \eqref{eq:C_n(x)} and
        \eqref{eq:fs_emp} (red ribbon), and the middle 95th percentile
        (purple ribbon)
        of all the replicates each closely agree.  The minor
        differences in the confidence intervals for the right tail of
        the lifetime distribution are eliminated as the sample size
        increases (row two versus row one).
        The only deterioration occurs in the very left tail of the
        lifetime distribution (bottom, left), which is a result of truncation
        causing very few simulated observations (the point estimate is still
        quite accurate).  The results in the bottom row  
        used a sample size of $n = 10{,}000$.  All
        results used 1{,}000 replicates.}
    \label{fig:CI_comps}
\end{figure}

We next summarize the approximation accuracy across various sample sizes,
and we observe it is a function of the underlying distribution.  
Intuitively, this can also be gleaned from
\eqref{eq:thm3_sigma_simp};
the variance of the estimator $\hat{\lambda}_n$ is a function of the
distributions of $X$ and $Y$.  Hence, we see that a larger sample size is
necessary to control the approximation accuracy towards the right tail
of the distribution of $X$, values of which occur with much smaller
probability.  We see minor tail failures of the approximation begin to
materialize when $n$ is as large as 1,000.  On the other hand, the
approximation for $\hat{\beta}_n$ still works well for
$n=1{,}000$.  See Figure~\ref{fig:CI_comps} for details.

Table~\ref{tab:cov_prob} summarizes the observed coverage probability
over the 1,000 replicates for various sample sizes, $n$.  That is,
the percentage of the 1,000 replicates of confidence intervals that contained
the true value of $\lambda(x)$, $x \in \{1, \ldots, 24\}$ and $\beta(y)$,
$y \in \{1, \ldots, 10\}$.  We also track the number
of replicates that did not return a valid estimate (i.e., we did not observe
any samples of $X$ or $Y$ at a particular value).  Given these results,
we recommend that a practitioner use judgment and available references
to estimate the probability of less frequent observations.  The smaller
these probabilities, generally speaking, the larger the sample to ensure
the approximation works well.  Alternatively, a practitioner may instead
identify the values of $X$ or $Y$ that are of most interest.  For example, the
confidence interval approximation for $\hat{\lambda}_n$ still works very
well for $X \leq 10$ when $n = 1,000$.  More details may be
found in Table~\ref{tab:cov_prob}. Alternatively,
if a known accurate estimate of $f_*$ and $\lambda$ is
available, then determining the appropriate sample size is only a matter of
selecting an approach to handle a simultaneous confidence region.

\begin{table}[tbh!]
  \centering
	\caption{Coverage percentages (CP) of 95\% confidence
          intervals under various sample sizes in the
          simulation study adjusted for the frequency of unrealized
          simulations (UR). Top: ${\lambda}_n(x)$ for
          $x \in \{1, \ldots, 24\}$; Bottom: ${\beta}_n(y)$, for
          $y \in \{1, \ldots, 10\}$.}
	\label{tab:cov_prob}
\begin{tabular}{ccccccccccc}
\toprule
& \multicolumn{2}{c}{$n = 250$} & \multicolumn{2}{c}{$n = 500$}
& \multicolumn{2}{c}{$n = 750$} & \multicolumn{2}{c}{$n = 1{,}000$} &
\multicolumn{2}{c}{$n = 10{,}000$}\\
\midrule
$x$ & CP & UR & CP & UR & CP & UR & CP & UR & CP & UR\\
\midrule
1&95.9&0&94.4&0&93.8&0&93.8&0&95.6&0\\
2&96.2&0&94.5&0&94.5&0&95.4&0&94.9&0\\
3&95.3&0&94.5&0&95.0&0&95.7&0&95.8&0\\
4&95.3&0&94.6&0&95.0&0&94.3&0&93.9&0\\
5&95.4&0&94.1&0&95.7&0&94.5&0&95.8&0\\
6&93.8&0&94.5&0&93.9&0&94.1&0&95.8&0\\
7&93.7&0&94.4&0&95.6&0&95.6&0&95.4&0\\
8&94.5&0&95.3&0&95.7&0&95.2&0&95.2&0\\
9&95.3&0&95.1&0&95.8&0&94.4&0&93.0&0\\
10&95.1&0&95.5&0&95.9&0&96.7&0&93.3&0\\
11&95.2&0&95.4&0&94.4&0&94.4&0&93.8&0\\
12&95.9&0&95.6&0&95.3&0&95.7&0&95.1&0\\
13&94.6&1&94.9&0&94.5&0&95.7&0&96.2&0\\
14&96.0&0&96.3&0&95.7&0&94.6&0&94.6&0\\
15&95.5&6&94.6&0&94.0&0&95.3&0&95.7&0\\
16&95.1&23&96.8&2&93.7&0&94.5&0&94.4&0\\
17&95.5&37&96.6&1&95.7&0&95.2&0&95.1&0\\
18&94.5&77&96.0&5&96.8&1&95.0&0&94.8&0\\
19&93.1&131&94.7&21&96.1&3&95.2&0&94.8&0\\
20&92.6&204&95.0&43&95.3&4&95.3&2&95.8&0\\
21&91.1&296&95.3&69&95.1&20&95.6&1&95.0&0\\
22&90.0&347&93.6&146&94.3&49&95.7&20&94.7&0\\
23&87.5&431&91.3&206&93.8&84&95.5&35&94.8&0\\
\midrule
& \multicolumn{2}{c}{$n = 100$} & \multicolumn{2}{c}{$n = 250$}
& \multicolumn{2}{c}{$n = 500$} & \multicolumn{2}{c}{$n = 1{,}000$} &
\multicolumn{2}{c}{$n = 10{,}000$}\\
\midrule
$y$ & CP & UR & CP & UR & CP & UR & CP & UR & CP & UR\\
\midrule
2&93.6&0&95.2&0&94.3&0&95.5&0&96.0&0\\
3&95.6&0&94.5&0&94.3&0&95.3&0&95.5&0\\
4&96.0&0&96.4&0&95.8&0&94.8&0&95.7&0\\
5&94.7&0&95.7&0&95.4&0&95.4&0&94.2&0\\
6&95.9&1&94.9&0&95.1&0&95.6&0&95.5&0\\
7&97.0&2&95.3&0&94.6&0&95.2&0&94.8&0\\
8&95.3&11&95.3&0&95.6&0&96.3&0&95.1&0\\
9&96.1&20&96.0&0&94.8&0&94.9&0&96.2&0\\
10&93.6&47&95.6&2&95.4&0&94.6&0&95.4&0\\
\midrule
\end{tabular}
\end{table}

For the second part of our simulation study, we consider a combination of a
lifetime random variable and a left-truncation random variable that is more
representative of an application to structured finance; specifically, leases
with an original termination schedule of 24 months.  The probabilities
are summarized in Figure~\ref{fig:appDist}.  We can see the
lifetime distribution obtains a peak near month 24, and the truncation 
distribution is not discrete uniform (compare with Figure~\ref{fig:mbalt}).
The
bottom row of Figure~\ref{fig:CI_comps} demonstrates experimental verification
of Theorems~\ref{thm:haz} and~\ref{thm:button} in this instance, as the
true hazard rates and estimates overlap (asymptotic unbiasedness) and the
true and empirical confidence intervals all closely agree (asymptotic
normality and independence).  The only deterioration in the estimator's
asymptotic performance occurs with the confidence intervals at the very left
tail of the lifetime distribution, which is a direct result of the combined
lifetime and left-truncation random variable probabilities causing heavy
left-truncation.  The sample size for each of the 1{,}000 replicates was
$n = 10{,}000$.

\begin{figure}[tbh!]
    \includegraphics[width=\textwidth]{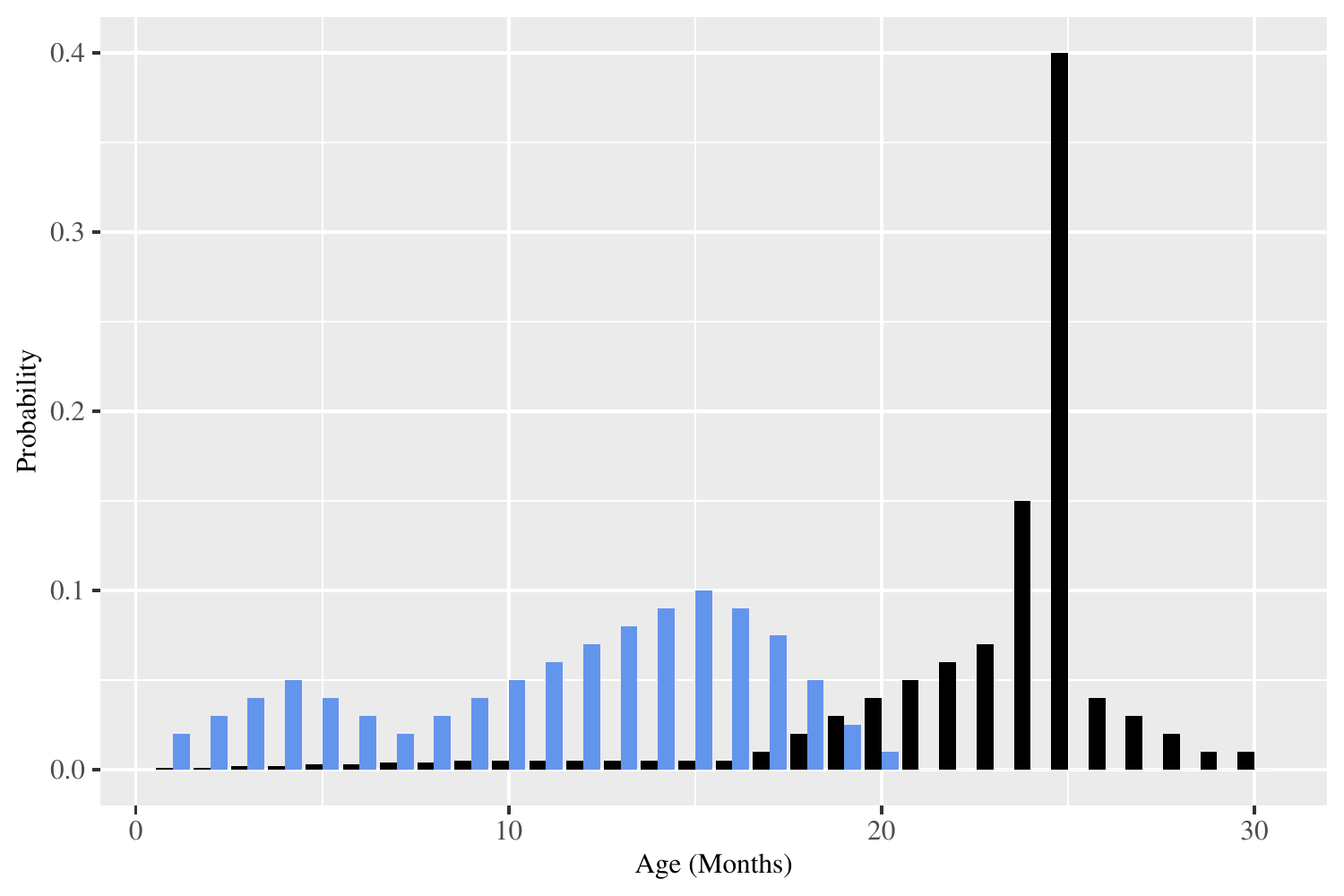}
        \caption{
The lifetime distribution (black bars) and left-truncation
distribution (blue bars) more representative of an
application to structured finance (compare with
Figure~\ref{fig:mbalt}) used to produce an additional
simulation verification of Theorems~\ref{thm:haz}
and \ref{thm:button} (bottom row, Figure~\ref{fig:CI_comps}).}
    \label{fig:appDist}
\end{figure}

\section{Application}
\label{sec:app}

Recall the motivating example in Section~\ref{sec:intro}. Here
we apply the estimation and asymptotic results of earlier sections to a
subset of auto lease securitization trust data.  Specifically, we examine the
Mercedes--Benz
Auto Lease Trust (MBALT) 2017-A financial transaction \citep{mercedes_2017}.
Detailed data and performance records are available at the individualized 
contract level from the Electronic Data Gathering, Analysis, and Retrieval 
(EDGAR) system, which is freely available to the public through the Securities
and Exchange Commission \citep{cfr_229}. The MBALT 2017-A transaction had
56,402 lease contracts with original terms ranging from 24 to 60 months.  For the
purposes of illustration, we only consider ongoing lease contracts with an original 
termination schedule of 24 months.  This reduced the sample to 866 lease 
contracts.

The MBALT 2017-A bond was placed in April of 2017.  The transaction was paid in
full and closed in August of 2019.  Therefore, the observation window consisted
of 28 months.  Monthly loan performance information is available on EDGAR.
Lease contracts must be delinquent no more than 30 days to be included in the
securitization trust \citep{mercedes_2017}.  Hence, the lease contracts are
all active as of the onset of the transaction.  At initialization, the oldest
lease in the trust was 21 months old, and the youngest lease was 3 months old.
Thus, to use our notation, $\Delta = 3$ and $m = 18$.
Though each lease is scheduled to terminate after 24 months, lease contracts
may terminate early through default or consumer option.  Additionally,
lease contracts may extend beyond 24 months due to missed payments or
various extension clauses.  Therefore, to estimate the
time of a lease termination, we searched the data for three consecutive months
of a zero payment.  Once three consecutive zero payments were found,
the month of lease termination was assigned to be the month of the first zero.
For example, if a lease contract recorded a zero payment for months 11, 12,
and 13, then month 11 was assumed to be the lease termination age.  After
performing this search, we identified eight contracts that did not terminate
during the observation window and were thus right-censored.  However, for simplicity,
we assumed these eight leases all terminated as of the last observation month.
(A related study, \cite{lautier_2022}, generalizes the estimators of
Section~\ref{sec:est} to the case of
right-censoring, see Section~\ref{sec:disc} for additional discussion.)
The termination time of the oldest lease was 37 months, and so $\omega = 37$.
Formally, then, $Y \in \{4, \ldots, 22\}$ and $X \in \{4,\ldots, 37\}$. (A
minor comment here is that we began counting $T$ at 0 within this application,
which is why the maximum bound of $Y$ extends to $m + \Delta + 1 = 22$.)

In Figure~\ref{fig:mbalt}, we plot the estimated hazard rate for 24-month 
leases within the MBALT 2017-A transaction.  Most of these leases have
terminated at month 25, which we would expect for a pool of leases
contractually designed to terminate after 24 payments.  However, there are
a few interesting observations.  First, there is notable early lease termination
activity beginning around lease age 20 months.  Second, we have sporadic hazard
rate behavior beyond lease age 25.  Finally, the width of the
95\% confidence band increases markedly beyond 25 months.  The bands
are quite narrow for leases that terminate prior to the original termination
schedule of 24 months, however.  
Table~\ref{tab:mbalt_res} presents complete results
for the estimated quantities $\hat{f}_{*,n}$, $\hat{\lambda}_n$,
$\hat{g}_{*,n}$, and $\hat{\beta}_n$, along with the standard errors for
$\hat{\lambda}_n$ and $\hat{\beta}_n$.

Additionally, some practitioners may be more interested in estimating the
left-truncation random variable, $Y$.  To this end, we present the estimated
probability mass function for $Y$ in Figure~\ref{fig:mbalt}.  An
interested investor could use this information to recover $T$, the
distribution of lease origination times.  Information about $T$ may be
compared with economic trends or the ``Selection of the Leases" section
of \citet{mercedes_2017}, for example.  Finally, it may be of interest
to determine if the distribution of $Y$ is discrete uniform, particularly if
one wishes to attempt to generalize these results into a length-biased
model, such as with \citet{asgharian_2002} and \citet{una_alvarez_2004}.
Though it may be obvious
from Figure~\ref{fig:mbalt} that $Y$ is not discrete uniform, we may also use
Corollary~\ref{cor:hypo} to calculate $\mathbb{Q}_U = 1{,}530.6$.  At
$q = \mathbf{card}\{5, \ldots, 22\} = 18$ degrees of freedom, this
corresponds to a $p$-value of effectively zero.  Hence,
we reject the null hypothesis of a discrete uniform distribution for $Y$.
Rejecting
the null in this case implies utilizing a method to estimate a distribution
function for $X$ that relies on the assumption that the left-truncation random
variable is discrete uniform (i.e., stationarity), such as length-biased
sampling, would be invalid for this application.

\begin{figure}[tbh!]
    \includegraphics[width=\textwidth]{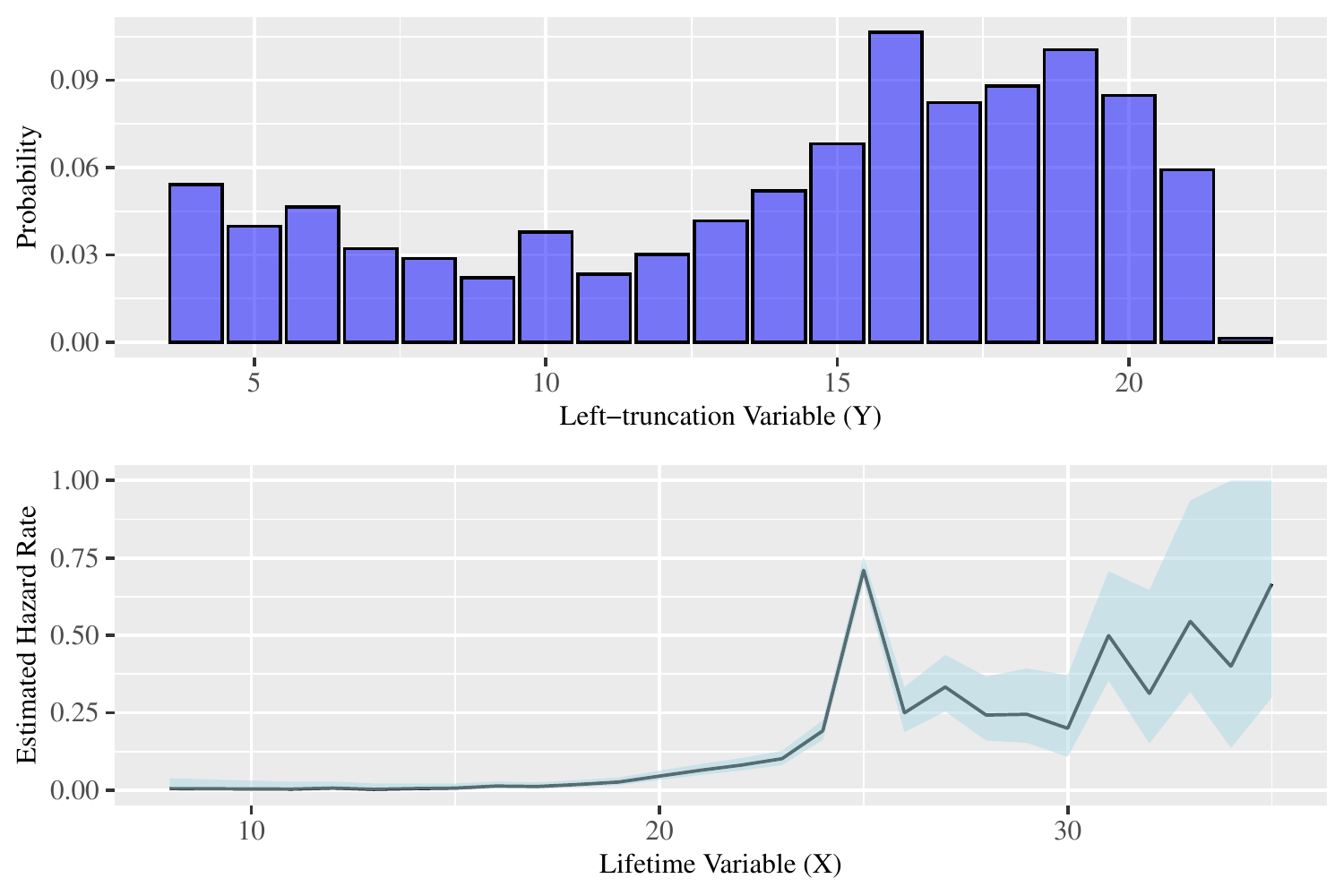}
        \caption{Summary plots for $\hat{g}_{*,n}$ (top) and 
        $\hat{\lambda}_n$ plus estimated 95\% confidence intervals
        (bottom) for a subset of 24-month leases from the MBALT
        2017-A securitization.}
    \label{fig:mbalt}
\end{figure}


\begin{table}[tbp]
  \centering
  \caption{Estimated distributions for the MBALT 2017-A application:
  the lifetime of interest (lease terminations, $F_0$) and the
  left-truncation random variable, $G_0$.}
\label{tab:mbalt_res}
\begin{tabular}{c ccc c ccc}
\toprule
\multicolumn{1}{c}{} & \multicolumn{3}{c}{$F_0$} & 
\multicolumn{1}{c}{} & \multicolumn{3}{c}{$G_0$}\\
\cmidrule(lr){2-4}\cmidrule(lr){6-8}
\multicolumn{1}{c}{Age} & 
\multicolumn{1}{c}{$\hat{f}_{*,n}$} &
$\hat{\lambda}_n$ & 
\multicolumn{1}{c}{$s.e.[ \hat{\lambda}_n ]$} &
Age &
$\hat{g}_{*,n}$ & 
$\hat{\beta}_n$ & 
$s.e.[ \hat{\beta}_n ]$\\
\midrule
4&0&0&0& 4 & 0.057&1&NA\\
5&0&0&0& 5 & 0.042&0.424&1.577\\
6&0&0&0& 6 & 0.048&0.331&1.229\\
7&0&0&0& 7 & 0.033&0.186&0.917\\
8&0.001& 0.005&0.161& 8 & 0.030&0.143&0.763\\
9&0&0&0& 9 & 0.023&0.100&0.621\\
10&0&0&0& 10 & 0.039&0.145&0.675\\
11&0.001&0.004&0.115& 11 & 0.024&0.082&0.505\\
12&0.002&0.007&0.147& 12 & 0.031&0.096&0.516\\
13&0.001&0.003&0.093& 13 & 0.043&0.117&0.531\\
14&0.002&0.006&0.115& 14 & 0.053&0.127&0.515\\
15&0.003&0.007&0.121& 15 & 0.069&0.143&0.502\\
16&0.008&0.014&0.152& 16 & 0.107&0.182&0.503\\
17&0.008&0.012&0.135& 17 & 0.082&0.124&0.404\\
18&0.014&0.019&0.157& 18 & 0.087&0.117&0.373\\
19&0.022&0.027&0.177& 19 & 0.097&0.118&0.355\\
20&0.040&0.046&0.223& 20 & 0.080&0.090&0.305\\
21&0.058&0.065&0.260& 21 & 0.053&0.059&0.250\\
22&0.068&0.081&0.298& 22 & 0.001&0.001&0.041\\
23&0.079&0.102&0.345&&&\\
24&0.133&0.192&0.474&&&\\
25&0.397&0.711&0.607&&&\\
26&0.040&0.250&1.077&&&\\
27&0.040&0.333&1.354&&&\\
28&0.020&0.243&1.508&&&\\
29&0.015&0.245&1.739&&&\\
30&0.009&0.200&1.861&&&\\
31&0.018&0.500&2.601&&&\\
32&0.006&0.313&3.410&&&\\
33&0.007&0.545&4.418&&&\\
34&0.002&0.400&6.447&&&\\
35&0.002&0.667&8.009&&&\\
36&0&0&0&&&\\
37&0.001&1&NA&&&\\
\midrule
\end{tabular}
\end{table}

\section{Discussion}
\label{sec:disc}

The estimates in Table~\ref{tab:mbalt_res} and Figure~\ref{fig:mbalt} may
have important applications for practitioners to understand the behavior of
consumer automobile leaseholders when given the option to terminate or extend
a lease contract.  Financially, risk professionals can use our estimation 
procedures to model the relationship between consumer lessee behavior and 
the credit risk of securitized bonds.  Automobile manufacturers may also have
an interest in our application in terms of modeling the relationship between
profitability and the structure of a consumer lease contract.  The connective
thread of this manuscript is that the estimates of Section~\ref{sec:app} were
produced using the theoretical results of Sections~\ref{sec:est}
and~\ref{sec:asym_res}.

To that end, this is to our knowledge the first thorough exposition of the case
of data subject to random left-truncation in the case of discrete $X$ 
and $Y$ with finite support.  We proved that the random estimation vectors
$\hat{\bm{\Lambda}}_n$ and $\hat{\mathbf{B}}_n$ are together an MLE for the
parameters of the conditional bivariate distribution $H_*$ and asymptotically
normal with independent components (i.e., a diagonal covariance matrix).
Both results utilized an alternative sampling and left-truncation framework
from \citet{woodroofe_1985}, which was necessary to appropriately mimic the
practicalities of consumer ABS data.  We also further proved asymptotic normality 
extends to the survival function estimator $\hat{S}_n$ and the distribution 
function estimator $\hat{G}_n$.  The last main result of this work was to 
establish a hypothesis test to examine the shape of the distribution of $G$, 
which has utility to formally test the stationarity assumption of 
the left-truncation distribution in length-biased sampling.

The practical realities of econometric data can inform statistical analysis,
and we have identified a large group of securitized financial data that 
suggests the use of a survival analysis model adjusted for discrete-time data 
over a fixed time horizon subject to random left-truncation.  However, many forms
of economic or financial data fall into the same criteria studied herein.
Indeed, payment history is often recorded on a periodic basis, such as monthly,
quarterly, or annually. For example, a monthly frequency is common for 
insurance products and debt instruments, such as insurance premiums, credit 
card payments, mortgages, auto loans, and so on. Further, many financial 
contracts typically have a fixed, finite term, such as any standard auto loan 
or term life insurance. Even whole life insurance, which is technically 
written with payments due in perpetuity is, in actuality, a fixed-length 
contract of unknown duration (one may comfortably cap assumed lifetimes at 130
years, for example).   Our contributions to the asymptotic statistical 
properties of the discrete distribution function estimators can be
applied to and further investigated in alternative applications, such
as those of insurance, mortgages, and other debt instruments.

Looking ahead, many applications of estimating a lifetime distribution random
variable from observed data will also be subject to the further incomplete data 
complication of right-censoring.  Interested readers may find generalizations 
of select theoretical results within this paper to the case of both 
left-truncation and right-censoring in \citet{lautier_2022}, which also
includes an extended financial pricing model and application to securitization
data that utilizes the associated distribution estimators.  In addition, general 
empirical economic analysis may benefit from the introduction of explanatory
variables or covariates, similar to the classical regression models for 
survival data \citep[e.g., ][Section 2.6]{klein_2003} but appropriately 
calibrated for the discrete-time setting of Section~\ref{sec:est}.  We leave 
this problem open to further research.  In addition, it is of theoretical
interest to consider a discrete lifetime distribution over countably infinite
values (i.e., extending trapezoid $A$ in Figure~\ref{fig:trapezoid} to the
right indefinitely).
Given our application to finite term financial products, however, we also
leave this problem open to further research.


\appendix
\section{Literature Review}
\label{sec:lit_review}

The following is a chronological review of related literature to the seminal
papers \citet{woodroofe_1985} and \citet{wang_1986} regarding the problem of
estimating a distribution function from left-truncated data.  The distribution 
functions of the random variable of interest, $X$, and the left-truncation 
random variable, $Y$, are denoted by $F$ and $G$, respectively.

\citet{chao_1988} further study the estimator of $F$ by expressing a hazard
process as i.i.d.\  means of random variables and imposing
the same conditions as \citet{woodroofe_1985}.  The result is the ability
to represent the difference of $F$ and its estimator as i.i.d.\  means of
random variables to obtain weak convergence, including
the associated covariance structures.  \citet{keiding_1990} reparametrize the
left-truncation model as a three-state Markov process to invoke the
statistical theory of counting processes by \citet{aalen_1978}
to establish the nonparametric maximum likelihood estimator (NPMLE),
consistency, asymptotic normality,
and efficiency.  Both papers derive results assuming continuity of $F$,
however. \citet{lai_1991} relax the continuity assumption of $F$ in using
martingale integral representations and empirical process theory to prove
uniform strong consistency and weak convergence results, though they modify
the product-limit estimator in doing so.

Somewhat more recently, \citet{gurler_1993} examine hazard
functions and their derivatives for nonparametric kernel estimators.
Similarly, they again assume continuity of $G$ in proving asymptotic normality.
\citet{stute_1993} derives an almost sure representation of the estimator for
$F$ with weaker distributional assumptions than \citet{woodroofe_1985} and
improved error bounds. \citet{chen_1995} prove the \citet{lynden_1971}
estimator is uniformly strong consistent over the whole half line, a problem
left open by \citet{woodroofe_1985}.  Both papers assume continuity of $F$ and
$G$ throughout.  In part one of a two-part sequence, \citet{he_1998b} find a
simpler representation for the estimator of the truncation probability to show
strong consistency and asymptotic normality via an i.i.d.\  representation.  While,
these results are true for arbitrary $F$ and $G$, they do not consider the
estimators for the distribution functions for $F$ and $G$. In part two,
\citet{he_1998a} prove that the estimator for $F$ obeys the strong law of
large numbers when estimating $F_0$ for arbitrary and not necessarily
continuous $F$ (recall the distinction between $F$ and $F_0$
in Section~\ref{sec:est}). This relaxes the assumption of continuity but does
not address asymptotic normality.

The classical problem of estimating $F$ from truncated data has
by now become commonplace in textbooks \citep[e.g.,][]{karr_1991, pena_1999,
owen_2001,hu_2013}, but any extended treatment assumes continuity of $F$
\citep[e.g.,][\S5.5.3]{pena_1999}.

Finally, we expanded our review to consider the random left-truncation
model along with right-censoring.  A seminal work in this field is
\citet{tsai_1987}, which gives asymptotic results when left-truncated data are 
also subject to right-censoring.  Nonetheless, the authors also assume 
continuous $F$. The continuity of $F$ and $G$ is assumed in related works
\citep{uzong_1992, gijbels_1993, gurler_1996, zhou_1996,
zhou_1999, asgharian_2005, huang_2011}.

\section{Complete Proofs}
\label{sec:proofs}

\subsection{Proof of Theorem~\ref{thm:MLE}}
\label{sec:proofs:mle}

\begin{proof}
Without loss of generality, assume $\Delta = 0$.  For convenience of
notation, let $f_u \equiv f(u)$, $g_v \equiv g(v)$.  Then, restating 
\eqref{eq:h(u,v)_extended} in terms of the sampled pairs from $h_*$,
$(X_i, Y_i)$, $1 \leq i \leq n$, we have
\begin{equation*}
    h_*(u,v) = \Pr(X_i = u, Y_i = v) =
    \frac{ f_u g_v }{ \alpha }, 
    \quad u,v \in \mathcal{A},
\end{equation*}
with the accompanying extended definition
\begin{equation}
    \alpha = \Pr(Y \leq X) 
    = \sum_{u=1}^{\omega} f_u  \bigg( \sum_{v=1}^{\min(u,m)} g_v \bigg)
    = \sum_{v = 1}^{m} g_v \bigg( \sum_{u=v}^{\omega} f_u \bigg).
    \label{eq:alpha}
\end{equation}
Therefore, the
quantities $0 < f_u < 1$, $u \in \mathcal{A}$, and $0 < g_v < 1$, 
$v \in \mathcal{A}$ are the parameters to be estimated.
The shape of $h_*$ over $\mathcal{A}$ with this parametric interpretation
continues to have complete flexibility, 
and so this is an alternative interpretation of a nonparametric 
estimation problem under the setting of Section~\ref{sec:est}.  Since we are
working with a probability space, we must have 
$\sum_{u} f_u = \sum_{v} g_v = 1$.  This implies there are 
$(\omega - 1) + (m - 1)$ free parameters.

Denoting $\bm{f} = (f_1, \ldots, f_{\omega})^{\top}$
and $\bm{g} = (g_1, \ldots, g_m)^{\top}$, the likelihood and loglikelihood
are then,
\begin{equation*}
    L \big( \bm{f}, \bm{g} \mid \{(X_i, Y_i)\}_{1 \leq i \leq n} \big)
    = \prod_{v=1}^{m} \prod_{u=v}^{\omega} 
    \bigg[ \frac{ f(u) g(v) }{\alpha} \bigg]^
    { \textstyle \sum_{i=1}^{n} \mathbf{1}_{(X_i,Y_i)=(u,v)}},
\end{equation*}
and
\begin{equation}
    l(\bm{f}, \bm{g}) \equiv
    \frac{1}{n} \log 
    L \big( \bm{f}, \bm{g} \mid \{(X_i, Y_i)\}_{1 \leq i \leq n} \big)
    = - \log \alpha + \sum_{v=1}^{m} \sum_{u=v}^{\omega} \hat{h}_{vu}
    \{ \log f_u + \log g_v \},
    \label{eq:log_like_fg}
\end{equation}
where
\begin{equation*}
    \hat{h}_{vu} = \frac{1}{n} \sum_{i=1}^{n} \mathbf{1}_{(X_i,Y_i)=(u,v)}.
\end{equation*}
As is standard procedure, our goal is to maximize \eqref{eq:log_like_fg}.
There are two ways to formulate this problem.  The first is as a constrained
optimization.  Specifically, the parameter space of $\bm{f}$ and $\bm{g}$
is the $m \times \omega$ dimensional hypercube over the unit interval
$\mathcal{I} = (0,1)$, and we seek
\begin{equation}
\bigg{ \{ }
    \max_{ \bm{f}, \bm{g} }
    l(\bm{f}, \bm{g})
    :
    \sum_{u = 1}^{\omega} f_u = 1;
    \sum_{v=1}^{m} g_v = 1; 
    \underset{1 \leq u \leq \omega}{f_u},
    \underset{1 \leq v \leq m}{g_v} \in \mathcal{I}
    \bigg{ \} }.
    \label{eq:max_eq}
\end{equation}
That is, $l(\bm{f}, \bm{g}): (0,1)^{m \times \omega} \mapsto \mathbb{R}$,
subject to the constraints in \eqref{eq:max_eq}.  It is not straightforward
to see that any solution will be a global maximum, however.

Alternatively, we can restrict the domain of $l(\bm{f}, \bm{g})$ to the
convex set
\begin{equation*}
    \Psi = \bigg{ \{ }
    \underset{1 \leq u \leq \omega}{f_u},
    \underset{1 \leq v \leq m}{g_v} \in \mathcal{I} :
    \sum_{u=1}^{\omega} f_u = \sum_{v = 1}^{m} g_v = 1
    \bigg{ \} }.
\end{equation*}
To see that $\Psi$ is convex, without loss of generality, let 
$0 \leq \varphi \leq 1$ and suppose 
$f^*_u = \varphi f'_{u} + (1 - \varphi)f''_{u}$ for
$f'_{u}, f''_{u} \in \Psi$ and $u \in \mathcal{A}$. Then
\begin{align*}
    \sum_{u=1}^{\omega} f^*_u 
    &= \sum_{u=1}^{\omega} \{ \varphi f'_{u} + (1 - \varphi)f''_{u} \}\\
    &= \varphi \sum_{u=1}^{\omega} f'_{u} 
    + (1 - \varphi) \sum_{u=1}^{\omega} f''_{u}\\
    &= 1,
\end{align*}
and $f^*_u \in \Psi$.
Thus, $l(\bm{f}, \bm{g}): \Psi \mapsto \mathbb{R}$, and, from the convexity of
$\Psi$, it is sufficient to claim we have found a global maximum if we can show
$l(\bm{f}, \bm{g})$ has only one stationary point that is not on the boundary
of $\Psi$.

A point on the boundary of $\Psi$ implies that there exists at least one
$f_u = 0$ or $g_v = 0$ for $u,v \in \mathcal{A}$.  But, this immediately
implies \eqref{eq:log_like_fg} explodes to negative infinity, (we assume here
$\alpha > 0$ to avoid the degenerate case of complete data loss; see also the
stricter conditions on $\hat{f}_{*,n}$ and $\hat{g}_{*,n}$ in the statement of
Theorem~\ref{thm:MLE}).  Hence, the maximum of \eqref{eq:log_like_fg} cannot
lie on the boundary of $\Psi$, and, if we can show $l(\bm{f}, \bm{g})$ has only
one stationary point, we can be assured it is a global maximum (i.e., the MLE).

We now show the system of partial derivatives with respect to each parameter
equated to zero has a single, unique solution.  In the following, that
$u, v \in \mathcal{A}$, i.e., $u,v \in \mathbb{N}$, is left assumed but will
be dropped for ease of presentation.  Observe first from \eqref{eq:alpha},
\begin{equation*}
    \frac{ \partial \alpha }{ \partial f_u } = \sum_{v=1}^{\min(u,m)} g_v,
    \quad \text{ and } \quad
    \frac{ \partial \alpha }{ \partial g_v } = \sum_{u=v}^{\omega} f_u.
\end{equation*}
Hence,
\begin{equation}
    \frac{\partial l(\bm{f}, \bm{g})}{\partial g_v}
    = \frac{1}{g_v} \sum_{u=v}^{\omega} \hat{h}_{vu} 
    - \frac{1}{\alpha} \frac{ \partial \alpha }{ \partial g_v }
    = 0, \quad 1 \leq v \leq m,
    \label{eq:part_gv}
\end{equation}
and
\begin{equation}
    \frac{\partial l(\bm{f}, \bm{g})}{\partial f_u}
    = \frac{1}{f_u} \sum_{v=1}^{\min(u,m)} \hat{h}_{vu}
    - \frac{1}{\alpha} \frac{ \partial \alpha }{ \partial f_u }
    = 0, \quad 1 \leq u \leq \omega.
    \label{eq:part_fu}
\end{equation}
The simultaneous solution to \eqref{eq:part_gv} and \eqref{eq:part_fu}
may be determined sequentially.  We proceed by mathematical induction.  
That is, for $v = 1$, with \eqref{eq:part_gv},
\begin{equation*}
    \frac{1}{g_1} \sum_{u=1}^{\omega} \hat{h}_{1u}
    - \frac{1}{\alpha} \sum_{u=1}^{\omega} f_u = 0
    \implies
    \hat{g}_1 = \alpha \sum_{u=1}^{\omega} \hat{h}_{1u}
    = \alpha \hat{C}_n(1).
\end{equation*}
Thus, for $u = 1$, with \eqref{eq:part_fu},
\begin{equation*}
    \frac{1}{f_1} \sum_{v=1}^{1} \hat{h}_{v1}
    - \frac{1}{\alpha} \sum_{v=1}^{1} \hat{g}_v = 0
    \implies
    \hat{f}_1 = \frac{ \hat{h}_{11} }{ \hat{C}_n(1) }
    = \frac{ \frac{1}{n} \sum_{i=1}^{n} \mathbf{1}_{X_i=1} }{\hat{C}_n(1)}
    = \hat{\lambda}_n(1).
\end{equation*}
Consider now $v= 2$ with \eqref{eq:part_gv},
\begin{equation*}
    \frac{1}{g_2} \sum_{u=2}^{\omega} \hat{h}_{2u}
    - \frac{1}{\alpha} \sum_{u=2}^{\omega} f_u = 0
    \implies
    \hat{g}_2 = \frac{ \alpha \sum_{u=2}^{\omega} \hat{h}_{2u} }{1 - \hat{f}_1}
    = \frac{ \alpha \sum_{u=2}^{\omega} \hat{h}_{2u} }{1 - \hat{\lambda}_n(1)}.
\end{equation*}
Thus, for $u = 2$, with \eqref{eq:part_fu}
\begin{align*}
    0 &= \frac{1}{f_2} \sum_{v=1}^{2} \hat{h}_{v2} -
    \frac{1}{\alpha} \sum_{v=1}^{2} \hat{g}_v\\
    &= \frac{1}{f_2} \sum_{v=1}^{2} \hat{h}_{v2} -
    \frac{1}{\alpha} \bigg[ 
    \alpha \hat{C}_n(1) + 
    \frac{ \alpha \sum_{u=2}^{\omega} \hat{h}_{2u} }{1 - \hat{\lambda}_n(1)}
    \bigg]\\
    &= \frac{1}{f_2} \sum_{v=1}^{2} \hat{h}_{v2} -
    \bigg[
    \frac{\hat{C}_n(1) - \sum_{i=1}^{n} \mathbf{1}_{X_i} + 
    \sum_{u=2}^{\omega} \hat{h}_{2u}}
    {1 - \hat{\lambda}_n(1)}
    \bigg]\\
    &= \frac{1}{f_2} \sum_{v=1}^{2} \hat{h}_{v2} -
    \frac{ \hat{C}_n(2) }{1 - \hat{\lambda}_n(1)}.
\end{align*}
That is,
\begin{equation*}
    \hat{f}_2 = \hat{\lambda}_n(2) [1 - \hat{\lambda}_n(1)].
\end{equation*}
Now assume the induction hypothesis for $1 \leq k < m$; i.e.,
\begin{equation*}
    \hat{g}_k = \frac{ \alpha \sum_{u=k}^{\omega} \hat{h}_{ku} }
    {1 - \sum_{j=1}^{k-1} \hat{f}_j },
    \quad \text{ and } \quad
    \hat{f}_k = 
    \hat{\lambda}_n(k) \prod_{1 \leq j < k}[1 - \hat{\lambda}_n(j)],
\end{equation*}
with the conventions $\sum_{j=1}^{0} \hat{f}_j = 0$ and
$\prod_{1 \leq j < 1} [1 - \hat{\lambda}_n(j)] = 1$.  Then by
\eqref{eq:part_gv},
\begin{equation*}
    \frac{1}{g_{k+1}} \sum_{u = k + 1}^{\omega} \hat{h}_{k+1u}
    - \frac{1}{\alpha} \sum_{u = k + 1}^{\omega} f_u
    \implies
    \hat{g}_{k+1} = 
    \frac{ \alpha \sum_{u = k + 1}^{\omega} \hat{h}_{k+1u}}
    {1 - \sum_{j=1}^{k} \hat{f}_j}.
\end{equation*}
But, for $1 \leq r \leq k$,
\begin{align}
    1 - \sum_{j=1}^{r} \hat{f}_j &=
    1 - \hat{\lambda}_n(1) - \hat{\lambda}_n(2)[1 - \hat{\lambda}_n(1)]
    - \cdots
    - \hat{\lambda}_n(r)[1 - \hat{\lambda}_n(r-1)] \cdots [1 - \hat{\lambda}_n(1)]
    \nonumber\\
    &= \prod_{j=1}^{r} [1 - \hat{\lambda}_n(j)].
    \label{eq:f_ident}
\end{align}
Thus,
\begin{equation*}
    \hat{g}_{k+1} = 
    \frac{ \alpha \sum_{u = k + 1}^{\omega} \hat{h}_{k+1u}}
    {\prod_{j=1}^{k} [1 - \hat{\lambda}_n(j)]}.
\end{equation*}
Therefore, for $u = k + 1$, with \eqref{eq:part_fu},
\begin{equation*}
    \frac{1}{f_{k+1}} \sum_{v=1}^{k+1}\hat{h}_{vk+1} - 
    \frac{1}{\alpha} \sum_{v=1}^{k+1} \hat{g}_v = 0.
\end{equation*}
Further,
\begin{align*}
    \sum_{v=1}^{k+1} \hat{g}_v 
    ={}& \alpha \hat{C}_n(1) + 
    \frac{ \alpha \sum_{u=2}^{\omega} \hat{h}_{2u} }{1 - \hat{\lambda}_n(1)}
    + \cdots +
    \frac{ \alpha \sum_{u=k}^{\omega} \hat{h}_{ku} }
    {\prod_{j=1}^{k-1} [1 - \hat{\lambda}_n(j)]}
    + \frac{ \alpha \sum_{u = k + 1}^{\omega} \hat{h}_{k+1u}}
    {\prod_{j=1}^{k} [1 - \hat{\lambda}_n(j)]}\\
    ={}& \frac{ \alpha \hat{C}_n(2) }{1 - \hat{\lambda}_n(1)} +
    \frac{ \alpha \sum_{u=3}^{\omega} \hat{h}_{3u} }
    {\prod_{j=1}^{2} [1 - \hat{\lambda}_n(j)]} + 
    \cdots + 
    \frac{ \alpha \sum_{u=k}^{\omega} \hat{h}_{ku} }
    {\prod_{j=1}^{k-1} [1 - \hat{\lambda}_n(j)]} +
    \frac{ \alpha \sum_{u = k + 1}^{\omega} \hat{h}_{k+1u}}
    {\prod_{j=1}^{k} [1 - \hat{\lambda}_n(j)]}\\
    ={}& \frac{ \alpha \hat{C}_n(3) }{\prod_{j=1}^{2} [1 - \hat{\lambda}_n(j)]} +
    \cdots +
    \frac{ \alpha \sum_{u=k}^{\omega} \hat{h}_{ku} }
    {\prod_{j=1}^{k-1} [1 - \hat{\lambda}_n(j)]} +
    \frac{ \alpha \sum_{u = k + 1}^{\omega} \hat{h}_{k+1u}}
    {\prod_{j=1}^{k} [1 - \hat{\lambda}_n(j)]}\\
    &\vdots\\
    ={}& \frac{ \alpha \hat{C}_n(k) }
    { \prod_{j=1}^{k-1} [1 - \hat{\lambda}_n(j)] } +
    \frac{ \alpha \sum_{u = k + 1}^{\omega} \hat{h}_{k+1u}}
    {\prod_{j=1}^{k} [1 - \hat{\lambda}_n(j)]}\\
    ={}& \frac{ \alpha \hat{C}_n(k+1) }{\prod_{j=1}^{k} [1 - \hat{\lambda}_n(j)]}.
\end{align*}
That is,
\begin{equation*}
    \hat{f}_{k+1} = \hat{\lambda}_n(k+1) \prod_{j=1}^{k} [1 - \hat{\lambda}_n(j)].
\end{equation*}
Now, for $m < u \leq \omega$,
\begin{equation*}
    \frac{1}{f_u} \sum_{v=1}^{m} \hat{h}_{vu} -
    \frac{1}{\alpha} \sum_{v=1}^{m} \hat{g}_v = 0
    \iff
    \hat{f}_u = \frac{1}{n} \sum_{i=1}^{n} \mathbf{1}_{X_i = u} 
    \frac{ \prod_{j=1}^{m-1} [1 - \hat{\lambda}_n(j)] }
    {\hat{C}_n(m)}.
\end{equation*}
Hence,
\begin{align*}
\hat{f}_u&=
    \frac{1}{n} \sum_{i=1}^{n} \mathbf{1}_{X_i = u} 
    \frac{ \prod_{j=1}^{m-1} [1 - \hat{\lambda}_n(j)] }
    {\hat{C}_n(m)}\\
    &=
    \frac{ \frac{1}{n} \sum_{i=1}^{n} \mathbf{1}_{X_i = u} }
    {\hat{C}_n(u)}
    \frac{ \hat{C}_n(u) }{ \hat{C}_n(u-1) }
    \frac{ \hat{C}_n(u-1)}{ \hat{C}_n(u-2) }
    \cdots
    \frac{ \hat{C}_n(m+1) }{ \hat{C}_n(m) } 
    \prod_{j=1}^{m-1} [1 - \hat{\lambda}_n(j)]\\
    &= \hat{\lambda}_n(u)
    \bigg[
    \frac{\hat{C}_n(u-1) - \frac{1}{n} \sum_{i=1}^{n} \mathbf{1}_{X_i = u-1}}
    {\hat{C}_n(u-1)}\bigg] 
    \cdots
    \bigg[ \frac{\hat{C}_n(m) - \sum_{i=1}^{n} \mathbf{1}_{X_i = m}}{
    \hat{C}_n(m)}\bigg]
    \prod_{j=1}^{m-1} [1 - \hat{\lambda}_n(j)]\\
    &= \hat{\lambda}_n(u) \prod_{j=1}^{u-1} [1 - \hat{\lambda}_n(j)].
\end{align*}
Lastly, since $\hat{\lambda}_n(\omega) = 1$,
\begin{align*}
    \sum_{u=1}^{\omega} \hat{f}_u
    ={}& \sum_{u=1}^{\omega} \bigg( 
    \hat{\lambda}_n(u) \prod_{j=1}^{u-1} [1 - \hat{\lambda}_n(j)]
    \bigg)\\
    ={}& \hat{\lambda}_n(1) + (1 - \hat{\lambda}_n(1))
    \sum_{u=2}^{\omega} \bigg( \hat{\lambda}_n(u)
    \prod_{j=2}^{u-1} [1 - \hat{\lambda}_n(j)] \bigg)\\
    ={}& \hat{\lambda}_n(1) + (1 - \hat{\lambda}_n(1))[\cdots
    (1 - \lambda_n(\omega-2))
    [ \hat{\lambda}_n(\omega-1) + 1 - \hat{\lambda}_n(\omega-1)]\\
    {}&\vdots\\
    ={}& \hat{\lambda}_n(1) + 1 - \hat{\lambda}_n(1)\\
    ={}& 1,
\end{align*}
and the solution set $\hat{f}_u$, $u \in \mathcal{A}$, is in $\Psi$ and omits
only this single, unique solution.  It is thus the global maximum of
\eqref{eq:log_like_fg} and therefore the MLE.  More specifically, 
we have found the MLE for the parameters $f_u$, $u \in \mathcal{A}$,
and they are of the form \eqref{eq:f_to_lam}.
Therefore,  $\hat{\bm{\Lambda}}_n$ is an MLE of $f_u$, for $u \in \mathcal{A}$
by the invariance property of the MLE
\citep[e.g.,][Theorem 7.2.1, pg. 350]{nitis_2000}.

We can show $\hat{\mathbf{B}}_n$ is also a MLE for $g_v$, $v \in \mathcal{A}$,
by moving sequentially
from the other direction; e.g., for $m \leq k \leq \omega$, with
\eqref{eq:part_fu},
\begin{equation*}
    \frac{1}{f_{k}} \sum_{v=1}^{m} \hat{h}_{v k} -
    \frac{1}{\alpha} \sum_{v=1}^{m} g_v = 0
    \implies
    \hat{f}_{k} = \alpha \sum_{v=1}^{m} \hat{h}_{v k},
\end{equation*}
and thus, for $v = m$
\begin{equation*}
    \frac{1}{g_m} \sum_{u=m}^{\omega} \hat{h}_{um}
    - \frac{1}{\alpha} \sum_{u = m}^{\omega} \hat{f}_u = 0
    \implies \hat{g}_m = 
    \frac{ \alpha \sum_{u=m}^{\omega} \hat{h}_{um} }
    { \sum_{u = m}^{\omega} \hat{f}_u }
    = \frac{ \frac{1}{n} \sum_{i=1}^{n} \mathbf{1}_{Y_i = m} }
    {\hat{C}_n(m)} = \hat{\beta}_n(m).
\end{equation*}
The remainder follows through symmetry.
\end{proof}

\subsection{Proof of Lemma~\ref{thm:Cn}}

\begin{proof}
Observe
\begin{equation}
  \hat{\mathbf{C}}_n = 
  	\begin{bmatrix} 
  	\displaystyle 
  	\frac{1}{n} \sum_{i=1}^{n} \mathbf{1}_{Y_i \leq \Delta + 1 \leq X_i}\\
  	\vdots \\ 
  	\displaystyle
    \frac{1}{n} \sum_{i=1}^{n} \mathbf{1}_{Y_i \leq \omega \leq X_i} 
    \end{bmatrix}
  	=
  	\frac{1}{n} \sum_{i=1}^{n} 
  	\begin{bmatrix} Y_{\Delta + 1 (i)} \\
    \vdots \\
    Y_{\omega (i)} \end{bmatrix}, 
    \label{eq:Chat_WLLN}
\end{equation}
where $Y_{k(i)}$, $\Delta + 1 \leq k \leq \omega$ are i.i.d.\ Bernoulli
random variables with probability of
success given by
$\Pr(Y_i \leq k \leq X_i) = \Pr(Y \leq k \leq X \mid Y \leq X) = C(k)$ 
for $k = \Delta + 1, \ldots, \omega$.  Thus, $\text{E}[Y_{k(i)}] = C(k)$ and
$\text{Var}[Y_{k(i)}] = C(k) (1 - C(k))$.  Now, since
\[
  \mathbf{1}_{Y_i \leq
    k' \leq X_i} \mathbf{1}_{Y_i \leq k \leq X_i} = \mathbf{1}_{Y_i \leq
    \min(k',k), X_i \geq \max(k'k)},
\]
we have
\begin{align}
  \text{E}[Y_{k'(i)} Y_{k(i)}]
  = \text{E}[\mathbf{1}_{Y_i \leq \min(k',k), X_i \geq
  \max(k'k)}]
    = c(k',k), \label{eq:EY2}
\end{align}
for $k', k = \Delta + 1, \ldots, \omega$.  Thus,
\begin{align*}
    \text{Cov}[Y_{k'(i)} Y_{k(i)} ] &= \text{E}[Y_{k'(i)} Y_{k(i)}] - \text{E}[Y_{k'(i)}]
    \text{E}[Y_{k(i)}]\\
    &= c(k',k) - C(k')C(k).
\end{align*}
Recall that~\eqref{eq:EY2} reduces to $C(k)$ when $k' = k$.  The result then
follows by the multivariate Central Limit Theorem (CLT)
\citep[][Theorem 8.21, pg. 61]{lehmann_1998}.
\end{proof}

\subsection{Proof of Lemma~\ref{lem:Cn}}

\begin{proof}
Applying the Weak Law of Large Numbers
\citep[][Theorem 8.2, pg. 54-55]{lehmann_1998}
to~\eqref{eq:Chat_WLLN} gives us the result.
\end{proof}

\subsection{Proof of Theorem~\ref{thm:haz}}

For convenience of notation, let
\begin{align}
    r(u,v) &= \Pr(X_i = \max(u,v), Y_i \leq \min(u,v)) \nonumber\\
    &= \Pr(X = \max(u,v), Y \leq \min(u,v) \mid Y \leq X) \nonumber\\
    &= \sum_{y = \Delta + 1}^{\min(u,v)} h( \max(u,v), y) \nonumber\\
    &= \frac{1}{\alpha}\Pr(X = \max(u,v))\Pr(Y \leq \min(u,v)).
    \label{eq:r(u,v)}
\end{align}
Notice $r(z,z) = f_*(z)$ and $r(u,v) = r(v,u)$.

\begin{proof}
Recall~\eqref{eq:lam_hat}--\eqref{eq:C_n(x)} and observe
\begin{align*}
    \hat{\bm{\Lambda}}_n - \bm{\Lambda} &= 
    \begin{bmatrix}
    \hat{\lambda}_n(\Delta + 1) \\ \vdots \\ \hat{\lambda}_n(\omega) 
    \end{bmatrix}
    - 
    \begin{bmatrix} \lambda(\Delta + 1) \\ \vdots \\ \lambda(\omega)
    \end{bmatrix}
    = 
    \begin{bmatrix} 
    \displaystyle 
    \frac{1}{n}\sum_{i=1}^{n} 
    \frac{\mathbf{1}_{X_i = \Delta + 1}}{\hat{C}_n(\Delta+1)} - 
    \frac{ f_*(\Delta + 1) }{ C(\Delta + 1) } \\ 
    \vdots \\ 
    \displaystyle 
    \frac{1}{n}\sum_{i=1}^{n}
    \frac{ \mathbf{1}_{X_i = \omega}}{\hat{C}_n(\omega)}
     - \frac{ f_*(\omega) }{ C(\omega) }
    \end{bmatrix}\\
    &\\
    &= \mathbf{A}_n \times \frac{1}{n} \sum_{i=1}^{n} 
    \begin{bmatrix}
    Z_{\Delta + 1 (i)} \\ \vdots \\ Z_{\omega(i)}
    \end{bmatrix},
\end{align*}
where, for $\Delta + 1 \leq k \leq \omega$,
\begin{equation*}
    Z_{k(i)} = \mathbf{1}_{X_i = k}C(k) - \mathbf{1}_{Y_i \leq k \leq
    X_i}f_*(k),
\end{equation*}
and $\mathbf{A}_n = \text{diag}([\hat{C}_n(\Delta+1)C(\Delta+1)]^{-1}, \ldots,
[\hat{C}_n(\omega)C(\omega)]^{-1})$.  That is,
\begin{equation*}
\hat{\bm{\Lambda}}_n - \bm{\Lambda} = \mathbf{A}_n \times \frac{1}{n}
\sum_{i=1}^{n} \mathbf{Z}_{(i)},
\end{equation*}
where $\mathbf{Z}_{(i)} = (Z_{\Delta+1(i)}, \ldots, Z_{\omega(i)})^{\top}$,
$1 \leq i \leq n$ are i.i.d.\ random vectors.
We will also subsequently show that the components of random vector
$\mathbf{Z}_{(i)}$ are uncorrelated.

More specifically, $\mathbf{1}_{X_i = x}$
is a Bernoulli random variable with probability of success $f_*(x)$ and,
similarly, $\mathbf{1}_{Y_i \leq x \leq X_i}$ is a Bernoulli random variable
with probability of success $C(x)$.  Thus,
\begin{equation*}
    \text{E}[Z_{k(i)}] = f_*(k)C(k) - C(k)f_*(k) = 0.
\end{equation*}
Therefore,
\begin{align}
  &\text{Cov}[Z_{k(i)}Z_{k'(i)}] \\ 
  ={}& E \bigg[ \bigg( \mathbf{1}_{X_i = k}C(k) - \mathbf{1}_{Y_i
       \leq k \leq X_i}f_*(k) \bigg)
       \bigg( \mathbf{1}_{X_i = k'}C(k') -
       \mathbf{1}_{Y_i \leq k' \leq X_i}f_*(k') \bigg) \bigg] \nonumber\\
  ={}& C(k)C(k')\text{E}[\mathbf{1}_{X_i = k} \mathbf{1}_{X_i = k'}] -
       f_*(k)C(k')\text{E}[\mathbf{1}_{X_i = k'}\mathbf{1}_{Y_i \leq k \leq X_i}]
       \nonumber\\
  &-C(k)f_*(k')\text{E}[\mathbf{1}_{X_i = k}\mathbf{1}_{Y_i \leq k' \leq X_i}]
    + f_*(k)f_*(k')\text{E}[\mathbf{1}_{Y_i \leq k \leq X_i} \mathbf{1}_{Y_i \leq k'
    \leq X_i}]. \label{eq:Cov_thm3}
\end{align}
We proceed to calculate $\text{Cov}[Z_{k(i)}Z_{k'(i)}]$ by cases.

Case 1: $k = k'$.

Notice
$\mathbf{1}_{X_i = k} \mathbf{1}_{X_i = k'} =
\mathbf{1}_{X_i = k}$ and
$\text{E}[\mathbf{1}_{X_i = k} \mathbf{1}_{X_i = k'}] = f_*(k)$.  Further,
\begin{align*}
  \mathbf{1}_{X_i = k'}\mathbf{1}_{Y_i \leq k \leq X_i}
  = \mathbf{1}_{X_i = k, Y_i \leq k \leq X_i}
  = \mathbf{1}_{X_i = k}.
\end{align*}
Hence, $\text{E}[\mathbf{1}_{X_i = k}\mathbf{1}_{Y_i \leq k' \leq X_i}] = f_*(k)$.
Also note that
\begin{equation*}
\mathbf{1}_{Y_i \leq k \leq X_i}\mathbf{1}_{Y_i \leq k'
\leq X_i} = \mathbf{1}_{Y_i \leq k \leq X_i},
\end{equation*}
and thus $\text{E}[\mathbf{1}_{Y_i \leq k \leq X_i}] = C(k)$.
Replacing the expectations in~\eqref{eq:Cov_thm3} yields
\begin{align}
\text{Cov}[Z_{k(i)}Z_{k'(i)}] ={}& C(k)C(k')f_*(k) - f_*(k)C(k')f_*(k)\nonumber\\
&-C(k)f_*(k')f_*(k) + f_*(k)f_*(k')C(k) \nonumber\\
={}& C(k)^2 f_*(k) - 2 f_*(k)^2 C(k) + f_*(k)^2 C(k) \nonumber\\
={}& f_*(k) C(k) [ C(k) - f_*(k) ]. \label{eq:Cov_thm3_long}
\end{align}
However,
\begin{align}
C(k) - f_*(k) &= \sum_{y=\Delta+1}^{k} \sum_{x=k}^{L} h_*(x,y) -
\sum_{y=\Delta+1}^{k}h_*(k,y) \nonumber\\
&= \sum_{y=\Delta+1}^{k} \sum_{x=k+1}^{L} h_*(x,y) \nonumber\\
&= c(k, k+1). \label{eq:Ck-f*k}
\end{align}
Replacing~\eqref{eq:Ck-f*k} in~\eqref{eq:Cov_thm3_long} and simplifying
yields the diagonal matrix
\begin{equation*}
\mathbf{D} = \text{diag} \big( f_*(\Delta + 1)C(\Delta+1)c(\Delta + 1,
\Delta + 2), \ldots, f_*(\omega) C(\omega) c(\omega,\omega+1) \big).
\end{equation*}
We emphasize here that $c(\omega,\omega+1) = 0$.

Case 2: $k \neq k'$.

Certainly, $\mathbf{1}_{X_i = k} \mathbf{1}_{X_i = k'} = 0$ when $k
\neq k'$.  Therefore,
\begin{equation}
    \text{E}[\mathbf{1}_{X_i = k} \mathbf{1}_{X_i = k'}] = 0. \label{eq:E_left}
\end{equation}
Assume $k < k'$ and notice $\mathbf{1}_{X_i = k'}\mathbf{1}_{Y_i \leq k
\leq X_i} = \mathbf{1}_{X_i = k', Y_i \leq k \leq X_i}$.  Thus,
$\text{E}[\mathbf{1}_{X_i = k'}\mathbf{1}_{Y_i \leq k \leq X_i}] = r(k',k)$.
On the other hand, $\mathbf{1}_{X_i = k}\mathbf{1}_{Y_i \leq k' \leq X_i}
= \mathbf{1}_{X_i = k, Y_i \leq k' \leq X_i} = 0$ because $\{X_i = k \cap
k' \leq X_i\} = \emptyset$ when $k < k'$.  Now observe the symmetry between
$\mathbf{1}_{X_i = k}\mathbf{1}_{Y_i \leq k' \leq X_i}$ and $\mathbf{1}_{
X_i = k'}\mathbf{1}_{Y_i \leq k \leq X_i}$ to drop the assumption
$k < k'$ and more generally claim
\begin{align}
    &-f_*(k)C(k')\text{E}[\mathbf{1}_{X_i = k'}\mathbf{1}_{Y_i \leq k \leq X_i}] - C(k)f_*(k')\text{E}[\mathbf{1}_{X_i = k}\mathbf{1}_{Y_i \leq k' \leq X_i}]
    \nonumber\\
    = {} & -r(k,k')f_*(\min(k,k'))C(\max(k,k')). \label{eq:E_mid}
\end{align}
Further, $\mathbf{1}_{Y_i \leq k \leq X_i} \mathbf{1}_{Y_i \leq k' \leq X_i}
= \mathbf{1}_{Y_i \leq k \leq X_i, Y_i \leq k' \leq X_i} = \mathbf{1}_{Y_i
\leq \min(k,k'), X_i \geq \max(k,k')}$.  Hence,
\begin{equation}
    \text{E}[\mathbf{1}_{Y_i \leq k \leq X_i} \mathbf{1}_{Y_i \leq k' \leq X_i}]
    = c(k,k'). \label{eq:E_right}
\end{equation}
Replacing the expectations~\eqref{eq:E_left}, \eqref{eq:E_mid}, and
\eqref{eq:E_right} in~\eqref{eq:Cov_thm3} and simplifying yields
\begin{equation*}
\text{E}[Z_{k(i)}Z_{k'(i)}] = f_*(\min(k,k')) \times \{ f_*(\max(k,k'))c(k,k') -
r(k,k')C(\max(k,k')) \}.
\end{equation*}
But,
\begin{align*}
&\phantom{ = }\, f_*(\max(k,k'))c(k,k')\\
&= \frac{\Pr(X = \max(k,k'),Y \leq X)}{\alpha}
\frac{ \Pr(Y \leq \min(k,k')) \Pr(X \geq \max(k,k'))}{\alpha} \nonumber\\
&= \frac{ \Pr(X = \max(k,k')) \Pr(Y \leq \max(k,k')) }{ \alpha }\frac{ \Pr(Y
\leq \min(k,k')) \Pr(X \geq \max(k,k'))}{\alpha} \nonumber\\
&= \frac{ \Pr(X = \max(k,k')) \Pr(Y \leq \min(k,k'))}{\alpha} \frac{ \Pr(Y
\leq \max(k,k')) \Pr(X \geq \max(k,k'))}{\alpha} \nonumber\\
&= r(k,k') C(\max(k,k')),
\end{align*}
and so~\eqref{eq:Cov_thm3} is zero whenever $k \neq k'$. Now define
\begin{equation*}
\bar{\mathbf{Z}}_n = \frac{1}{n} \sum_{i=1}^{n} \mathbf{Z}_{(i)},
\end{equation*}
and use the multivariate CLT
\citep[][Theorem 8.21, pg. 61]{lehmann_1998} to claim
\begin{equation*}
    \sqrt{n}[\bar{\mathbf{Z}}_n - \bm{0}] \overset{\mathcal{L}}{\longrightarrow}
    N(\bm{0}, \mathbf{D}), \text{ as } n \rightarrow \infty.
\end{equation*}
Further note by Lemma~\ref{lem:Cn},
\begin{equation*}
\mathbf{A}_n \overset{\mathcal{P}}{\longrightarrow} \mathbf{V},
\text{ as } n \rightarrow \infty
\end{equation*}
where $\mathbf{V} = \text{diag}(C(\Delta+1)^{-2}, \ldots, C(\omega)^{-2})$.
Therefore, by multivariate Slutsky's Theorem
\citep[][Theorem 5.1.6, pg. 283]{ELST},
\begin{equation*}
    \sqrt{n}[\mathbf{A}_n \bar{\mathbf{Z}}_n] \overset{\mathcal{L}}{
    \longrightarrow} N(\bm{0}, \mathbf{V} \mathbf{D} \mathbf{V}^{\top}),
     \text{ as } n \rightarrow \infty.
\end{equation*}
Finally, observe $\mathbf{VDV}^{\top} = \bm{\Sigma}_f$ and
$\mathbf{A}_n \bar{\mathbf{Z}}_n = \hat{\bm{\Lambda}}_n - \bm{\Lambda}$
to complete the proof.

\end{proof}

\subsection{Proof of Theorem~\ref{thm:button}}

\begin{proof}
See the proof of Theorem~\ref{thm:haz}, substituting $g_*$ for $f_*$ and
adjusting the indicator logic as appropriate.  It is useful to introduce
similar notation to~\eqref{eq:r(u,v)}.  That is,
\begin{align}
    s(u,v) &= \Pr(Y_i = \min(u,v), X_i \geq \max(u,v)) \nonumber\\
    &= \frac{1}{\alpha} \Pr(Y = \min(u,v)) \Pr(X \geq \max(u,v)).
    \label{eq:s(u,v)}
\end{align}

\end{proof}

\subsection{Proof of Theorem~\ref{thm:Sn}}

\begin{proof}
To motivate the demonstration, let $x \in \{\Delta+1, \ldots, \omega\}$
and recall \eqref{eq:F_haz} to write,
\begin{equation*}
    S(x) = \prod_{z = \Delta + 1}^{x} [1 - \lambda(z)].
\end{equation*}
Now consider the natural log,
\begin{equation*}
    \ln S(x) = \sum_{z = \Delta + 1}^{x} \ln [1 - \lambda(z)].
\end{equation*}
Hence,
\begin{align*}
    \sqrt{n}[ \ln S_n(x) - \ln S(x) ] &= \sqrt{n} \bigg[ \sum_{z=\Delta+1}^{
    x} \ln \bigg( \frac{1 - \lambda_n(z)}{1 - \lambda(z)} \bigg) \bigg]
    \nonumber\\
    &= \sqrt{n} \bigg[ \sum_{z=\Delta + 1}^{x} \ln \bigg(1 + \frac{
    \lambda(z) - \lambda_n(z) }{1 - \lambda(z)} \bigg) \bigg].
\end{align*}
But $\ln(1 + x) = \sum_{n \geq 1}(-1)^{n+1} x^n / n$ and so
\begin{align}
    \sqrt{n}[ \ln S_n(x) - \ln S(x) ] &= \sqrt{n} \bigg[ \sum_{z=\Delta+1}^{
    x} \bigg{ \{ } \frac{ \lambda(z) - \lambda_n(z) }{1 - \lambda(z)} -
    \frac{1}{2} \bigg[ \frac{ (\lambda(z) - \lambda_n(z))^2 }{ (1 -
    \lambda(z))^2 } \bigg] + \cdots \bigg{ \} } \bigg] \nonumber\\
    &=\sqrt{n} \bigg[ \sum_{z=\Delta+1}^{x} \frac{ \lambda(z) -
    \lambda_n(z) }{ 1 - \lambda(z) } + O_p( | \lambda(z) - \lambda_n(z) |^2
    ) \bigg] \nonumber\\
    &= \sqrt{n} \bigg( - \sum_{z=\Delta+1}^{x} \frac{ \lambda_n(z) -
    \lambda(z) }{ 1 - \lambda(z) } \bigg) + o_p(1), \label{eq:thm5_lO}
\end{align}
where~\eqref{eq:thm5_lO} follows by Corollary~\ref{cor:Hn} and
Slutsky's Theorem \citep[][Theorem 8.10, pg. 58]{lehmann_1998}.
Now consider all $x \in \{\Delta+1, \ldots, \omega\}$ to write,
\begin{equation*}
    \sqrt{n} \begin{bmatrix} \{ \ln S_n(\Delta + 1) - \ln S(\Delta + 1) \} \\
    \vdots \\
    \{ \ln S_n(\omega) - \ln S(\omega) \}
    \end{bmatrix} = \mathbf{K} \times \sqrt{n} (\hat{\bm{\Lambda}}_n -
    \bm{\Lambda}) + o_p(1).
\end{equation*}
Thus, by Theorem~\ref{thm:haz} and multivariate Slutsky's Theorem
\citep[][Theorem 5.1.6, pg. 283]{ELST},
\begin{equation*}
  \mathbf{D} \times \sqrt{n} (\hat{\bm{\Lambda}}_n - \bm{\Lambda}) +
  o_p(1)
  \overset{\mathcal{L}}{\longrightarrow}
  N(0, \mathbf{K} \bm{\Sigma}_f \mathbf{K}^{\top}),
\text{ as } n \rightarrow \infty.
\end{equation*}
Finally, note $S(x) = \exp \{ \ln S(x) \}$ and apply the multivariate
delta method \citep[][Theorem 8.22, pg. 61]{lehmann_1998} to complete
the proof.

\end{proof}

\subsection{Proof of Theorem~\ref{thm:Gn}}

\begin{proof}
Recall~\eqref{eq:G_est} and see the proof of Theorem~\ref{thm:Sn}.

\end{proof}

\subsection{Proof of Theorem~\ref{thm:hypo_gen}}

\begin{proof}
Begin with Theorem~\ref{thm:button} along with~\eqref{eq:sig_g_2} and use the
well-known multivariate normal results: (1) all subsets of multivariate
normal random vectors have themselves a normal distribution 
\citep[][Result 5.2.8, pg. 154]{FCILM} and (2) a centered and scaled quadratic
form of a $p$ dimensional multivariate normal random vector is a chi-squared
random variable with $p$ degrees of freedom 
\citep[][Result 5.3.3, pg. 167]{FCILM}.
The result then follows by the continuous mapping theorem 
\citep[][Corollary 8.11, pg. 58]{lehmann_1998}.
\end{proof}

\subsection{Proof of Corollary~\ref{cor:hypo}}

\begin{proof}
By the Weak Law of Large Numbers
\citep[][Theorem 8.2, pg. 54-55]{lehmann_1998},
$\hat{g}_{*,n} \overset{\mathcal{P}}{\rightarrow}
g_*$.  Further, if $G$ is discrete uniform over
$\{\Delta+1, \ldots, \Delta+m\}$, then for 
$y \in \{\Delta + 1, \ldots, \Delta + m\}$
\begin{equation*}
\beta(y) 
= \frac{\Pr(Y = y)}{\Pr(Y \leq y)}
= \frac{1}{m} \frac{m}{y - (\Delta + 1) +1} = \frac{1}{y - \Delta}.
\end{equation*}
Finally, use the results of Theorem~\ref{thm:hypo_gen} substituting $\beta(y)$
for $y \in \{\Delta+2, \ldots, \Delta+m\}$ as appropriate along with
multivariate Slutsky's Theorem
\citep[][Theorem 5.1.6, pg. 283]{ELST} to complete the proof.

\end{proof}

\bibliographystyle{mcap}
\bibliography{refs}

\begin{thebibliography}{49}
\newcommand{\enquote}[1]{``#1''}
\expandafter\ifx\csname natexlab\endcsname\relax\def\natexlab#1{#1}\fi

\bibitem[{Aalen and Johansen(1978)}]{aalen_1978}
O.~O. Aalen and S.~Johansen (1978).
\newblock \enquote{An empirical transition matrix for non-homogeneous markov
  chains based on censored observations.}
\newblock {\em Scandinavian Journal of Statistics\/} {\bf 5}, 141--150.

\bibitem[{Addona and Wolfson(2006)}]{addona_2006}
V.~Addona and D.~B. Wolfson (2006).
\newblock \enquote{A formal test for the stationarity of the incidence rate
  using data from a prevalent cohort study with follow-up.}
\newblock {\em Lifetime Data Analysis\/} {\bf 12}, 267--284.

\bibitem[{Asgharian et~al.(2002)Asgharian, M'Lan and Wolfson}]{asgharian_2002}
M.~Asgharian, C.~E. M'Lan and D.~B. Wolfson (2002).
\newblock \enquote{Length-biased sampling with right censoring.}
\newblock {\em Journal of the American Statistical Association\/} {\bf 97},
  201--209.

\bibitem[{Asgharian and Wolfson(2005)}]{asgharian_2005}
M.~Asgharian and D.~B. Wolfson (2005).
\newblock \enquote{Asymptotic behavior of the unconditional {NPMLE} of the
  length-biased survivor function from right censored prevalent cohort data.}
\newblock {\em The Annals of Statistics\/} {\bf 33}, 2109--2131.

\bibitem[{Asgharian et~al.(2006)Asgharian, Wolfson and Zhang}]{asgharian_2006}
M.~Asgharian, D.~B. Wolfson and X.~Zhang (2006).
\newblock \enquote{Checking stationarity of the incidence rate using prevalent
  cohort survival data.}
\newblock {\em Statistics in Medicine\/} {\bf 25}, 1751--1767.

\bibitem[{Block et~al.(1998)Block, Savits and Singh}]{block_1998}
H.~W. Block, T.~H. Savits and H.~Singh (1998).
\newblock \enquote{The reversed hazard rate function.}
\newblock {\em Probability in the Engineering and Informational Sciences\/}
  {\bf 12}, 69–90.

\bibitem[{Chao and Lo(1988)}]{chao_1988}
M.-T. Chao and S.-H. Lo (1988).
\newblock \enquote{Some representations of the nonparametric maximum likelihood
  estimators with truncated data.}
\newblock {\em The Annals of Statistics\/} {\bf 16}, 661--668.

\bibitem[{Chen et~al.(1995)Chen, Chao and Lo}]{chen_1995}
K.~Chen, M.-T. Chao and S.-H. Lo (1995).
\newblock \enquote{On strong uniform consistency of the {L}ynden--{B}ell
  estimator for truncated data.}
\newblock {\em The Annals of Statistics\/} {\bf 23}, 440--449.

\bibitem[{de~la Pe\~na and Gin\'e(1999)}]{pena_1999}
V.~H. de~la Pe\~na and E.~Gin\'e (1999).
\newblock {\em Decoupling: From Dependence to Independence\/}.
\newblock Springer.

\bibitem[{De~U{\~n}a-{\'A}lvarez(2004)}]{una_alvarez_2004}
J.~De~U{\~n}a-{\'A}lvarez (2004).
\newblock \enquote{{Nonparametric estimation under length-biased sampling and
  Type I censoring: A moment based approach}.}
\newblock {\em Annals of the Institute of Statistical Mathematics\/} {\bf 56},
  667--681.

\bibitem[{Gijbels and Wang(1993)}]{gijbels_1993}
I.~Gijbels and J.~Wang (1993).
\newblock \enquote{Strong representations of the survival function estimator
  for truncated and censored data with applications.}
\newblock {\em Journal of Multivariate Analysis\/} {\bf 47}, 210--229.

\bibitem[{Guilbaud(1988)}]{guilbaud_1988}
O.~Guilbaud (1988).
\newblock \enquote{{Exact Kolmogorov-type tests for left-truncated and/or
  right-censored data}.}
\newblock {\em Journal of the American Statistical Association\/} {\bf 83},
  213--221.

\bibitem[{G\"urler(1996)}]{gurler_1996}
{\"U}.~G\"urler (1996).
\newblock \enquote{Bivariate estimation with right-truncated data.}
\newblock {\em Journal of the American Statistical Association\/} {\bf 91},
  1152--1165.

\bibitem[{G\"urler and Wang(1993)}]{gurler_1993}
{\"U}.~G\"urler and J.-L. Wang (1993).
\newblock \enquote{Nonparametric estimation of hazard functions and their
  derivatives under truncation model.}
\newblock {\em Annals of the Institute of Statistical Mathematics\/} {\bf 45},
  249--264.

\bibitem[{He and Yang(1998{\natexlab{a}})}]{he_1998b}
S.~He and G.~L. Yang (1998{\natexlab{a}}).
\newblock \enquote{Estimation of the truncation probability in the random
  truncation model.}
\newblock {\em The Annals of Statistics\/} {\bf 26}, 1011--1027.

\bibitem[{He and Yang(1998{\natexlab{b}})}]{he_1998a}
S.~He and G.~L. Yang (1998{\natexlab{b}}).
\newblock \enquote{The strong law under random truncation.}
\newblock {\em The Annals of Statistics\/} {\bf 26}, 992--1010.

\bibitem[{Hu(2013)}]{hu_2013}
C.~Hu (2013).
\newblock {\em Smoothing Spline ANOVA Models\/}.
\newblock Springer.

\bibitem[{Huang and Qin(2011)}]{huang_2011}
C.-Y. Huang and J.~Qin (2011).
\newblock \enquote{Nonparametric estimation for length-biased and
  right-censored data.}
\newblock {\em Biometrika\/} {\bf 98}, 177--186.

\bibitem[{Hwang and Wang(2008)}]{hwang_2008}
Y.-T. Hwang and C.-C. Wang (2008).
\newblock \enquote{A goodness of fit test for left-truncated and right-censored
  data.}
\newblock {\em Statistics \& Probability Letters\/} {\bf 78}, 2420--2425.

\bibitem[{Hyde(1977)}]{hyde_1977}
J.~Hyde (1977).
\newblock \enquote{{Testing survival under right censoring and left
  truncation}.}
\newblock {\em Biometrika\/} {\bf 64}, 225--230.

\bibitem[{Karr(1991)}]{karr_1991}
A.~F. Karr (1991).
\newblock {\em Point Processess and Their Statistical Inference\/}.
\newblock Marcel Dekker, Inc.

\bibitem[{Keiding and Gill(1990)}]{keiding_1990}
N.~Keiding and R.~D. Gill (1990).
\newblock \enquote{{Random truncation models and Markov processes}.}
\newblock {\em The Annals of Statistics\/} {\bf 18}, 582--602.

\bibitem[{Klein and Moeschberger(2003)}]{klein_2003}
J.~P. Klein and M.~L. Moeschberger (2003).
\newblock {\em Survival Analysis: Techniques for Censored and Truncated Data,
  Second Edition\/}.
\newblock Springer.

\bibitem[{Lai and Ying(1991)}]{lai_1991}
T.~L. Lai and Z.~Ying (1991).
\newblock \enquote{Estimating a distribution function with truncated and
  censored data.}
\newblock {\em The Annals of Statistics\/} {\bf 19}, 417--442.

\bibitem[{Lautier et~al.(2022)Lautier, Pozdnyakov and Yan}]{lautier_2022}
J.~P. Lautier, V.~Pozdnyakov and J.~Yan (2022).
\newblock \enquote{Modeling time-to-event contingent cash flows: A
  discrete-time survival analysis approach.}
\newblock ArXiv preprint, \url{https://arxiv.org/abs/2201.04981}.

\bibitem[{Lehmann and Casella(1998)}]{lehmann_1998}
E.~Lehmann and G.~Casella (1998).
\newblock {\em Theory of Point Estimation, 2nd Edition\/}.
\newblock Springer.

\bibitem[{Lehmann(1998)}]{ELST}
E.~L. Lehmann (1998).
\newblock {\em Elements of Large-Sample Theory\/}.
\newblock Springer.

\bibitem[{Lynden-Bell(1971)}]{lynden_1971}
D.~Lynden-Bell (1971).
\newblock \enquote{A method of allowing for known observational selection in
  small samples applied to {3CR} quasars.}
\newblock {\em Monthly Notices of the Royal Astronomical Society\/} {\bf 155},
  95--118.

\bibitem[{Mandel and Betensky(2007)}]{mandel_2007}
M.~Mandel and R.~A. Betensky (2007).
\newblock \enquote{Testing goodness of fit of a uniform truncation model.}
\newblock {\em Biometrics\/} {\bf 63}, 405--412.

\bibitem[{Mercedes-Benz(2017)}]{mercedes_2017}
Mercedes-Benz (2017).
\newblock \enquote{{Prospectus: Mercedes-Benz Auto Lease Trust 2017-A}.}
\newblock
  \url{https://www.sec.gov/Archives/edgar/data/1537805/000114036117016403/form424b2.htm}.
\newblock Online; accessed 24 February 2022.

\bibitem[{Moreira et~al.(2014)Moreira, De~U{\~n}a-{\'a}lvarez and
  Van~Keilegom}]{moreira_2014}
C.~Moreira, J.~De~U{\~n}a-{\'a}lvarez and I.~Van~Keilegom (2014).
\newblock \enquote{Goodness-of-fit tests for a semiparametric model under
  random double truncation.}
\newblock {\em Computational Statistics\/} {\bf 29}, 1365--1379.

\bibitem[{Mukhopadhyay(2000)}]{nitis_2000}
N.~Mukhopadhyay (2000).
\newblock {\em Probability and Statistical Inference\/}.
\newblock New York, NY: Marcel Dekker.

\bibitem[{Ning et~al.(2010)Ning, Qin and Shen}]{ning_2010}
J.~Ning, J.~Qin and Y.~Shen (2010).
\newblock \enquote{Non-parametric tests for right-censored data with biased
  sampling.}
\newblock {\em Journal of the Royal Statistical Society: Series B (Statistical
  Methodology)\/} {\bf 72}, 609--630.

\bibitem[{Owen(2001)}]{owen_2001}
A.~B. Owen (2001).
\newblock {\em Empirical Likelihood\/}.
\newblock Chapman \& Hall / CRC.

\bibitem[{Prentice and Gloeckler(1978)}]{prentice_1978}
R.~L. Prentice and L.~A. Gloeckler (1978).
\newblock \enquote{Regression analysis of grouped survival data with
  application to breast cancer data.}
\newblock {\em Biometrics\/} {\bf 34}, 57--67.

\bibitem[{Rabhi and Asgharian(2017)}]{rabhi_2017}
Y.~Rabhi and M.~Asgharian (2017).
\newblock \enquote{{Inference under biased sampling and right censoring for a
  change point in the hazard function}.}
\newblock {\em Bernoulli\/} {\bf 23}, 2720--2745.

\bibitem[{Ravishanker and Dey(2002)}]{FCILM}
N.~Ravishanker and D.~Dey (2002).
\newblock {\em A First Course in Linear Model Theory\/}.
\newblock Chapman \& Hall (CRC).

\bibitem[{{SEC}(2016)}]{cfr_229}
{SEC} (2016).
\newblock \enquote{{17 CFR \S 229.1125 (Item 1125) Schedule AL --- Asset-level
  information}.}
\newblock
  \url{https://www.govinfo.gov/app/details/CFR-2016-title17-vol3/CFR-2016-title17-vol3-sec229-1125}.
\newblock Online; accessed 24 February 2022.

\bibitem[{{SIFMA}(2022)}]{sifma_2022}
{SIFMA} (2022).
\newblock \enquote{{US ABS securities: Issuance, trading volume, outstanding}.}
\newblock
  \url{https://www.sifma.org/resources/research/us-asset-backed-securities-statistics/}.
\newblock Online; accessed 24 February 2022.

\bibitem[{Stute(1993)}]{stute_1993}
W.~Stute (1993).
\newblock \enquote{Almost sure representations of the product-limit estimator
  for truncated data.}
\newblock {\em The Annals of Statistics\/} {\bf 21}, 146--156.

\bibitem[{Tsai et~al.(1987)Tsai, Jewell and Wang}]{tsai_1987}
W.-Y. Tsai, N.~P. Jewell and M.-C. Wang (1987).
\newblock \enquote{{A note on the product-limit estimator under right censoring
  and left truncation}.}
\newblock {\em Biometrika\/} {\bf 74}, 883--886.

\bibitem[{Uzon{\={g}}ullari and Wang(1992)}]{uzong_1992}
{\"U}.~Uzon{\={g}}ullari and J.-L. Wang (1992).
\newblock \enquote{A comparison of hazard rate estimators for left truncated
  and right censored data.}
\newblock {\em Biometrika\/} {\bf 79}, 297--310.

\bibitem[{Vardi(1982)}]{vardi_1982}
Y.~Vardi (1982).
\newblock \enquote{Nonparametric estimation in the presence of length bias.}
\newblock {\em The Annals of Statistics\/} {\bf 10}, 616--620.

\bibitem[{Wang(1987)}]{wang_1987}
M.-C. Wang (1987).
\newblock \enquote{{Product limit estimates: A generalized maximum likelihood
  study}.}
\newblock {\em Communications in Statistics - Theory and Methods\/} {\bf 16},
  3117--3132.

\bibitem[{Wang(1991)}]{wang_1991}
M.-C. Wang (1991).
\newblock \enquote{Nonparametric estimation from cross-sectional survival
  data.}
\newblock {\em Journal of the American Statistical Association\/} {\bf 86},
  130--143.

\bibitem[{Wang et~al.(1986)Wang, Jewell and Tsai}]{wang_1986}
M.-C. Wang, N.~P. Jewell and W.-Y. Tsai (1986).
\newblock \enquote{Asymptotic properties of the product limit estimate under
  random truncation.}
\newblock {\em The Annals of Statistics\/} {\bf 14}, 1597--1605.

\bibitem[{Woodroofe(1985)}]{woodroofe_1985}
M.~Woodroofe (1985).
\newblock \enquote{Estimating a distribution function with truncated data.}
\newblock {\em The Annals of Statistics\/} {\bf 13}, 163--177.

\bibitem[{Zhou(1996)}]{zhou_1996}
Y.~Zhou (1996).
\newblock \enquote{A note on the {TJW} product-limit estimator for truncated
  and censored data.}
\newblock {\em Statistics \& Probability Letters\/} {\bf 26}, 381--387.

\bibitem[{Zhou and Yip(1999)}]{zhou_1999}
Y.~Zhou and P.~S.~F. Yip (1999).
\newblock \enquote{A strong representation of the product-limit estimator for
  left truncated and right censored data.}
\newblock {\em Journal of Multivariate Analysis\/} {\bf 69}, 261--280.

\end{thebibliography}

\end{document}